\documentclass[10pt]{article}
\usepackage{amsmath,amsthm,amsfonts,amssymb,amscd}
\usepackage[latin1]{inputenc}
\usepackage{xcolor}
\usepackage[mathcal]{eucal}
\usepackage{psfrag}
\usepackage{graphicx}
\usepackage{epstopdf}
\usepackage{xcolor} 
\theoremstyle{plain}
\newtheorem{lemma}{Lemma}[section]
\newtheorem{lema}[lemma]{Lemma}
\newtheorem{cor}[lemma]{Corollary}
\newtheorem{prop}[lemma]{Proposition}
\newtheorem{teo}[lemma]{Theorem}
\newtheorem{maintheorem}{Theorem}

\newcommand{\interior}{\operatorname{int}}

\theoremstyle{definition}

\newtheorem{remark}[lemma]{Remark}
\newtheorem{defi}[lemma]{Definition}

\title{Invariant Measures of Non-Uniformly Expanding Maps with Higher Order Critical Set} 
\author{Ricardo Chical\'e and Vanderlei Horita
\thanks{Work partially supported by FAPESP (2019/10269-3).}}
\date{\today}

\begin{document}

\maketitle

\let\thefootnote\relax
\let\thefootnote\relax

\begin{abstract}
We prove existence and uniqueness of absolutely continuous invariant measures for generalizations of Viana maps admitting a higher order critical point introduced in \cite{HMS24}.  
As a consequence of our approach, we obtain super-polynomial decay of correlations.
\end{abstract}

\section{Introduction}
In the study of non-uniformly expanding systems, Viana conjectured that a smooth map $f$ with only non-zero Lyapunov exponents at
Lebesgue almost every point has a physical measure, see \cite{Vi98}.
Among the motivations, let us mention the seminal work of Jakobson \cite{Ja81}, where he constructed absolutely continuous invariant probability (a.c.i.p. for short) for many quadratic maps of the interval having positive Lyapunov exponent.
In \cite{Vi97}, Viana introduced 2-dimensional skew-product maps coupling a quadratic map with a uniformly expanding circle map  presenting two positive Lyapunov exponents, currently known as Viana maps.
Alves in \cite{Al00} shows that Viana maps admit finitely many a.c.i.p.'s. 
In fact, Alves and Viana proved in \cite{AV02} the uniqueness of the measure. 
Alves, Bonatti, and Viana \cite{ABV00} proved existence of a finite numbem of a.c.p.i.'s for non-uniformly expanding local diffeomorphisms. This paper also shows the same for maps with singularities (i.e., maps that fails to be local diffeomorphism for some subset) having a condition of \emph{slow recurrence} of the orbits near the singular set.  


Horita, Muniz, and Sester \cite{HMS24} extend the result for Viana maps replacing the quadratic map $h(x) = a_0 - x^2 $ with a map $h_D$ with a non-flat critical point of any order. 
More precisely, let $\alpha> 0$ and $d\geq 16$ be real numbers and let $D\geq 2$ be a positive integer. 
Consider a $C^D$-map $\varphi_{\alpha,D}: \mathbb{S}^1\times \mathbb{R} \to \mathbb{S}^1\times\mathbb{R}$ of the following form 
$$
\varphi_{\alpha,D}(\theta,x) = (g(\theta),\alpha \sin(2\pi\theta)+h_D(x)),
$$
where $g \colon \mathbb{S}^1\rightarrow \mathbb{S}^1$ is the uniformly expanding map of the circle $\mathbb{S}^1 = \mathbb{R}/\mathbb{Z}$, $g(\theta) = d\theta \mod 1$, and $h_D$ is a map with an order $D$ critical point:
for the case when $D$ is even, $ D = 2q$, let us recall that $a_0$ in $(1,2)$ in Viana maps is taken for $x = 0$  to be a pre-periodic point of $h$, and since $a_0 < 2$ there exist a compact interval $I$ in $(-2,2)$ invariant by $h$. 
Consider two intervals $I', I'' \subset I$ such that $I''$ is a proper subinterval of $I' = (-1,1)$. 
We define $h_{2q} \colon I \to  I$ by
\begin{equation}
h_{2q}(x) = \left\{\begin{array}{lcl}
  a_0 - x^2 & \mbox{if} & x \in I\setminus I'\\
  a_0 - Ax^{2q} & \mbox{if} & x \in I'' 
\end{array}
\right. ,
\label{hd-even}
\end{equation}
where $A$ is a constant chosen such that the absolute value of the derivative of $h_{2q}$ at the extreme points of $I''$ are equal to $7/4$. 
Additionally, we require that in each component of $I'\setminus I''$ the first and second derivatives of $h_{2q}$ are monotone in $I\setminus I'$. 
So, $\tilde{x} = 0$ is the unique critical point of $h_{2q}$ and $h_{2q}(\tilde{x})$ is the fixed point of $h_{2q}$.
Moreover, $a_0$ is chosen in such a way that $h_{2q}^2(0)$ is the fixed point of $h_{2q}$.
\begin{figure}[!ht]
	\centering
	\includegraphics[width=5cm]{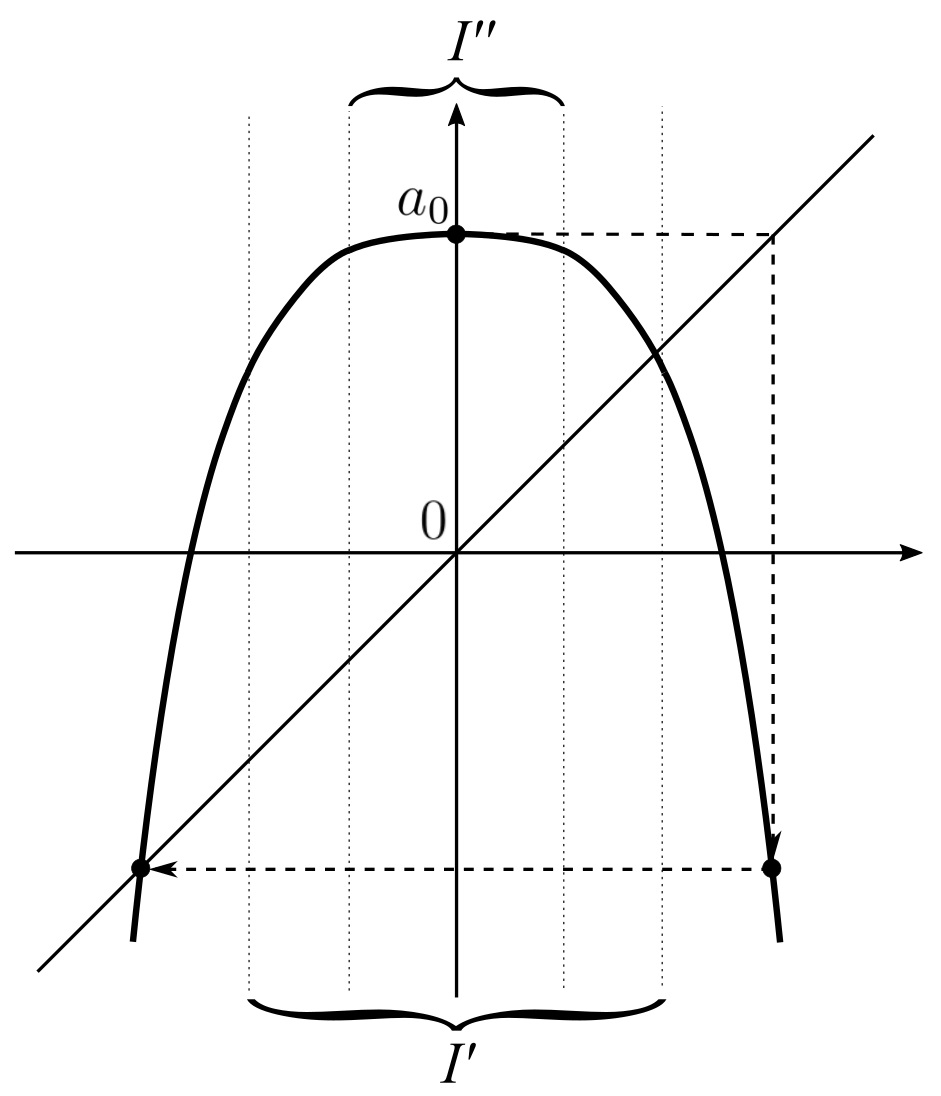}
	\caption{Map $h_D$ with a pre-periodic even critical point}
\end{figure}

For the case when $D$ is odd, $D= 2q+1$, let $I',I''\subset \mathbb{S}^1$ be intervals such that $I''$ is a proper subinterval of $I'$. 
One defines $h_{2q+1} \colon \mathbb{S}^1 \to \mathbb{S}^1$ by
\begin{equation}
	h_{2q+1}(x) = \left\{
	\begin{array}{lcl}
		2x \mod 1  & \mbox{if} & x\in \mathbb{S}^1\setminus I'\\
		A(x-1/2)^{2q+1} & \mbox{if} & x\in I''     
	\end{array}
	\right. .
	\label{hd-odd}
\end{equation}
Again, here we take $A$ as a constant such that the absolute value of the derivative of $h_{2q+1}$ at the extremal points of $I''$ are equal to $7/4$ and the first and second derivatives of $h_{2q+1}$ are monotone in $I' \setminus I''$. 
As defined, the map $h_{2q+1}$ has a unique critical point $\tilde{x} = 1/2$ of order $2q+1$. 
\begin{figure}[!ht]
	\centering
	\includegraphics[width=6cm]{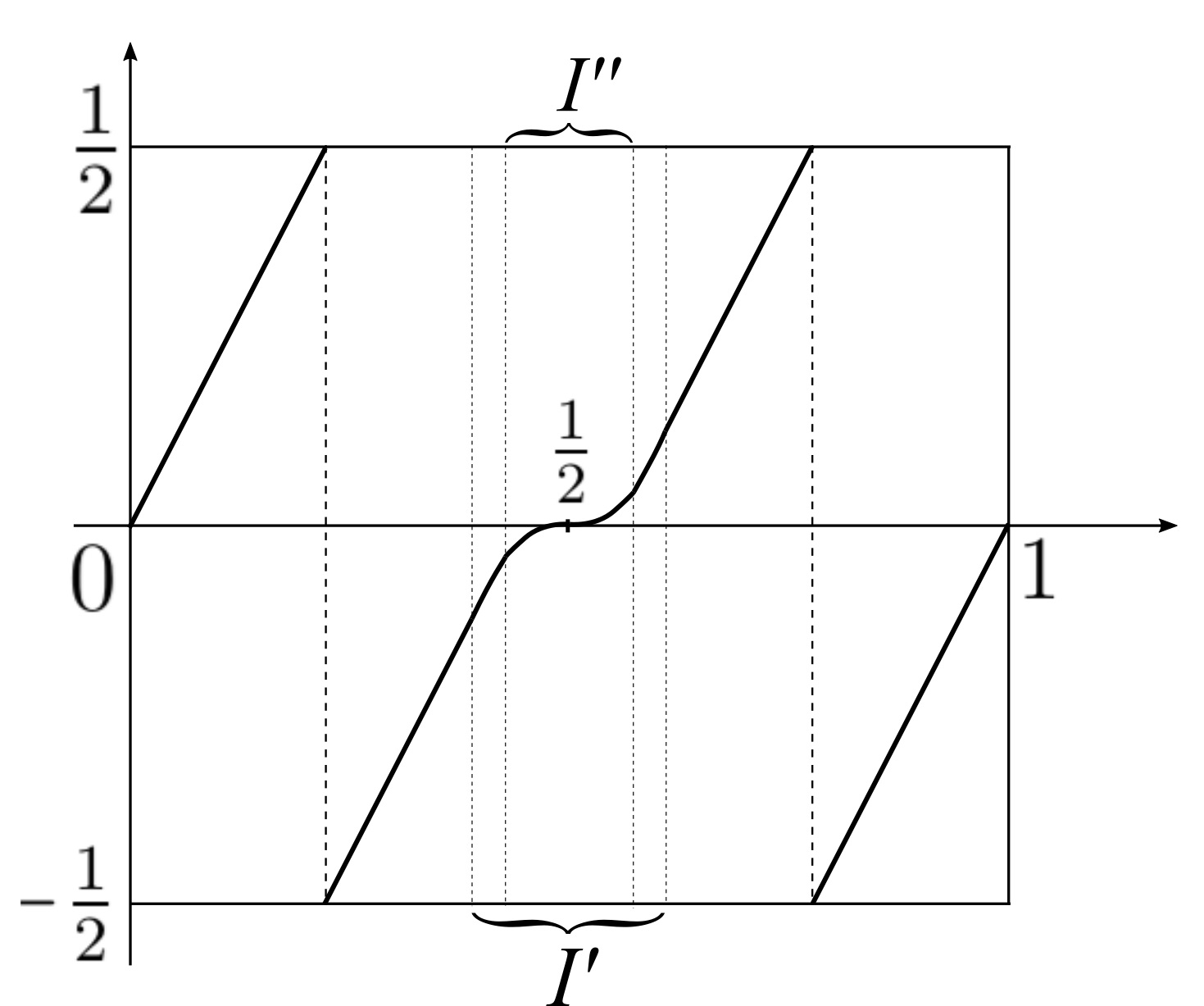}
	\caption{Map $h_D$ with an odd critical point}
\end{figure}

We write $\mathcal{M} = \mathbb{S}^1$ if $D$ is odd or $\mathcal{M} = I$ if $D$ is even.

In this work, we prove the existence and uniqueness of absolutely continuous invariant probabilities (a.c.i.p.'s) for these generalizations of Viana maps. Moreover, based on the work of Alves, Luzzatto, and Pinheiro in \cite{ALP05}, we conclude the decay of correlations for the corresponding dynamics and the Central Limit Theorem. First, we show the existence and finiteness of a.c.i.p.'s.

\begin{maintheorem}
For $d\geq 16$, $D\geq 2$ and $\alpha$ sufficiently small, the map $\varphi_{\alpha,D}$ has a finite invariant ergodic measure $\mu^{*}$ absolutely continuous with respect to the Lebesgue measure on $\mathbb{S}^1\times \mathcal{M}$ in every invariant component of the dynamics. Moreover, the same holds for every map $\varphi$ in a sufficiently small open neighborhood of $\varphi_{\alpha,D}$ in the $C^D(\mathbb{S}^1\times \mathcal{M})$ topology.    
\end{maintheorem}

The neighborhood mentioned in Theorem A will be the set

\begin{equation}
\mathcal{N} = \left\{ \varphi\in C^D(\mathbb{S}^1\times\mathcal{M}); \ ||\varphi - \varphi_{\alpha,D}|| < \alpha  \right\}.
\label{alpha-nbd}
\end{equation}

\begin{maintheorem}
Every map $\varphi\in\mathcal{N}$ is topologically mixing and admits a unique invariant ergodic measure absolutely continuous with respect to the Lebesgue measure.
\end{maintheorem}

Let $f \colon M \to M$ be a transformation and $\mu$ be an invariant probability. 
Recall that the \textit{correlation function} of a pair of functions $\phi,\psi \colon M \to \mathbb{R}$ is defined by
$$
C_n(\phi,\psi) = \left|\int_{X} (\phi\circ f^n)\psi \ d\mu - \int_X \phi \ d\mu \int_X \psi \ d\mu  \right|,
$$ 
whenever the integrals make sense.
The rapid decay of correlation suggests that a system may retain strong statistical properties, such as the \emph{Central Limit Theorem}: given a H\"older continuous function $\phi$ which is not a coboundary ($\phi \neq \psi \circ f - \psi$, for every $\psi$) there exists $\sigma > 0$ such that for every interval $\mathcal{I} \subset \mathbb{R}$, we have
$$
\mu \left( x \in M \colon \frac{1}{\sqrt{n}} \sum_{i=0}^{n-1} \left( \phi(f^i(x)) - \int \phi \; d\mu \right) \in \mathcal{I} \right) \to \frac{1}{\sigma \sqrt{2\pi}} \int_\mathcal{I} e^{-t^2}/2\sigma^2 \; dt.
$$ 

From Remark~\ref{rmk23}, the properties of the maps $\varphi \in \mathcal{N}$, and the Theorem 2 in \cite{ALP05}, we obtain the following result directly.

\begin{maintheorem}
For every map $\varphi\in \mathcal{N}$ and for H\"older continuous observables, the decay of correlations satisfies $C_n \le \mathcal{O}(n^{-\zeta})$ for every $\zeta > 0$. 
Moreover, the Central Limit Theorem holds for $\varphi$.
\end{maintheorem}

The construction of the a.c.i.p.'s is based on the approach developed by Alves \cite{Al00} and Alves-Viana \cite{AV02}. 
However, in the present setting, the presence of a higher degree critical point requires several changes and adaptations.
For instance, the construction of a partition for which we can obtain a piecewise uniformly expanding \emph{induced map} with bounded distortion, a key element to obtain a.c.i.p., see \cite[Theorem~5.2]{Al00}.

\section{Preliminary results and definitions}

We consider maps $\varphi: \mathbb{S}^1\times\mathcal{M}\rightarrow \mathbb{S}^1\times\mathcal{M}$ of the form
\begin{equation}
\varphi(\theta,x) = (g(\theta),f(\theta,x)), \ \mbox{with} \ \partial_x f(\theta,x) = 0 \ \mbox{if and only if} \ x= \tilde{x}
\label{cond1}
\end{equation}
and derive our results as long as $\varphi\in\mathcal{N}$.

To obtain the growth of the derivative of $\varphi$, we will study the returns of orbits to a neighborhood of the critical point.
Since in the $\theta$-direction the maps expands uniformly, we focus in the derivative of $f$ in the $x$-direction. 
Roughly speaking, as the critical point $\tilde{x}$ is pre-periodic, the points $x$ close to $\tilde{x}$ remains close to the orbit of $\tilde{x}$ for a \emph{large} amount of time. 
The time that the orbit of $x$ remains bind to the periodic point, where the derivative expands, permits to recover the lack of derivative near to the critical point.
On the other hand, while a orbit remains out of a neighborhood of the critical point, the derivative expands.
These is the heuristic present in part of arguments to build expansion for the map. 
For future references, we will summarize these contents in the next lemma, their proofs are in \cite[Lemmas 2.4 and 2.5]{HMS24}.

Given $(\theta,x)\in \mathbb{S}^1 \times \mathcal{M}$ we define $(\theta_i,x_i) = \varphi^i(\theta,x)$.
For the next lemma we take a positive constant $0 < \eta < 1/3$ depending only on the map $h_D$.
\begin{lema}
For every $\alpha > 0$ small enough, there exists an integer $N(\alpha)$ satisfying
\begin{itemize}
\item[(a)] If $|x-\tilde{x}| < 2\sqrt[D]{\alpha}$, then $\prod_{j=0}^{N(\alpha)-1} |\partial_x f(\theta_j,x_j)| \geq |x-\tilde{x}|^{D-1}\alpha^{-1+\frac{\eta}{D-1}}$.
\item[(b)] If $|x-\tilde{x}|<2\sqrt[D]{\alpha}$, then $|x_j-\tilde{x}|>\sqrt[D]{\alpha}$ for every $j = 1,\ldots,N(\alpha)$.
\item[(c)] $C_0\log(1/\alpha) \leq N(\alpha) \leq C_1\log(1/\alpha)$, for some constants $C_0,C_1>0$ .
\end{itemize}
There are $\sigma >1$, $C_2 >0$ and $\delta > 0$ such that for every $(\theta,x)\in \mathbb{S}^1 \times \mathcal{M}$ with $|x_0-\tilde{x}|,\ldots,|x_{k-1}-\tilde{x}| \geq \sqrt[D]{\alpha}$ and $k \ge1$
\begin{itemize}
\item[(d)] $\prod_{j=0}^{k-1} |\partial_x f(\theta_j,x_j)| \geq C_2 \sqrt[D]{\alpha^{D-1}} \sigma^k$ . 
\item[(e)] If, in addition, $|x_k-\tilde{x}| < \delta$ then $\prod_{j=0}^{k-1} |\partial_x f(\theta_j,x_j)| \geq C_2\sigma^k.$
\end{itemize}
\label{lema21'}
\end{lema}
The proof of itens (b) and (c) follows straightforward from the estimates in  \cite[Lemmas 2.4]{HMS24}, see also \cite[Lemma~2.1]{Al00}.

We now consider the full Lebesgue measure set of points $(\theta,x)\in \mathbb{S}^1\times \mathcal{M}$ that does not hit the critical set $\left\{ x = \tilde{x} \right\}$. 
We take the intervals $I_r = \left(\tilde{x} + \sqrt[D]{\alpha} e^{-r} ,    \tilde{x} + \sqrt[D]{\alpha}e^{-(r-1)}\right]$ for $r\geq 1$, and  $ I_r = \left[\tilde{x} - \sqrt[D]{\alpha} e^{r+1} , \tilde{x} - \sqrt[D]{\alpha}e^{r}\right)$ for $r\leq -1$.
For each $j\geq 0$, let 
$$
r_j(\theta,x) = \left\{
\begin{array}{ll}
|r|, & \mbox{if} \ \varphi^j(\theta,x)\in \mathbb{S}^1\times I_r \ \mbox{with} \ r \geq 1\\
0, & \mbox{if} \ \varphi^j(\theta,x)\in \mathbb{S}^1\times(\mathcal{M}\setminus I_r )
\end{array}  
\right. .
$$

We say that $v$ is a \emph{return situation} for $(\theta,x)$ if $r_v(\theta,x) \geq 1$. 
Given some positive integer $n$, let $0\leq v_1\leq\ldots\leq v_s\leq n$ be the return situations for $(\theta,x)$ from $0$ to $n-1$. 
Then, from Lemma \ref{lema21'} if follow that
\begin{align*}
\prod_{i = v_j}^{v_j+N(\alpha)-1}|\partial_x f(\theta_i,x_i)|  & \geq |x_{v_j}-\tilde{x}|^{D-1}\alpha^{-1+\frac{\eta}{D-1}} 
  \geq e^{-(D-1)r_{v_j}(\theta,x)}\alpha^{-1+\frac{D-1}{D}+\frac{\eta}{D-1}} \\
&  = e^{-(D-1)r_{v_j}(\theta,x)}\alpha^{-\frac{1}{D}+\frac{\eta}{D-1}},
\end{align*}
for every $j = 1,\ldots, s-1$.

Also from item (e) of Lemma \ref{lema21'}, for each $j = 1,\ldots, s-1$, we have the following estimates:
$$
\prod_{i=0}^{v_1-1} |\partial_x f(\theta_i,x_i)|  \geq C_2\sigma^{v_1} \quad \ \mbox{ and } \quad \ \prod_{i=v_j+N(\alpha)}^{v_{j+1}-1} |\partial_x f(\theta_i,x_i)| \geq C_2\sigma^{v_{j+1}-v_j-N(\alpha)}.$$



Finally, suppose $v_s = n$. Combining the three estimates above gives the following lower bound for $\log \prod_{i=0}^{n-1}|\partial_x f(\theta_i,x_i)|$: 


$$
[n-(s-1)N]\log \sigma +\sum_{k=1}^{s-1}\left[\left(\frac{1}{D}-\frac{\eta}{D-1} \right)\log\left(\dfrac{1}{\alpha}\right) - (D-1)r_{v_k}\right] - (s-1)\log(C_2).
$$

Consider 
$$
G_n(\theta,x) = \left\{0\leq v_i \leq n-1: \ r_{v_i}(\theta,x)\geq \dfrac{1}{D-1}\left(\dfrac{1}{D} -\dfrac{2\eta}{D-1} \right)\log\left(\dfrac{1}{\alpha}\right)  \right\}.
$$ 
Then, it follows from item (c) of Lemma \ref{lema21'}, that  
\begin{align*}
\sum_{k=1}^{s-1} \left[\left(\frac{1}{D}-\frac{\eta}{D-1} \right)  \log\left(\dfrac{1}{\alpha}\right) - (D-1)r_{v_k}\right] & \ge  
-(D-1)\sum_{i\in G_n} r_{i}(\theta,x) + \dfrac{(s-1)\eta}{D-1}\log\left(\dfrac{1}{\alpha}\right) \\
& \geq -(D-1)\sum_{i\in G_n} r_{i}(\theta,x) +\gamma N(\alpha) (s-1),
\end{align*}
for some constant $\gamma \leq \eta/(C_1(D-1))$.
 
 Now define $$c = \dfrac{1}{D+3}\min\left\{\gamma,\log\sigma \right\}.$$ 

From Lemma \ref{lema21'}, it follows $v_{j+1}-v_j \geq N(\alpha)$ for all $j$, which implies we must have 
\begin{equation}
s \leq \dfrac{n}{N(\alpha)}+1 \leq \dfrac{n}{C_0\log(1/\alpha)} +1
\label{eq9}
\end{equation} 
and choosing $\alpha$ small enough such that $\log C_2 \cdot \left[C_0\log(1/\alpha) \right]^{-1} \leq c$ we get $$s\log C_2 \leq cn +\log C_2.$$

Then our estimates for the lower bound become:
\begin{equation}
\log\prod_{i=0}^{n-1}|\partial_x f(\theta_i,x_i)| \geq (D+2) cn - (D-1)\sum_{i\in G_n} r_{i}(\theta,x),
\label{eq8}
\end{equation}
for $\alpha$ sufficiently small.

\begin{lema}
If $(\theta,x),(\tau,y)\in \mathbb{S}^1\times \mathcal{M}$ are points such that $r_j(\theta,x) \leq r_j(\tau,y) +4$ for every $j = 0,\ldots,n-1$, then 
$$
\prod_{i=1}^{n-1} |\partial_x f(\theta_i,x_i)| \geq \exp\left( (D+1)cn - (D-1)\sum_{j\in G_n(\tau,y)} r_{j}(\tau,y) \right) .
$$
\label{lemma23}
\end{lema}

\begin{proof}
By the hypothesis and the estimate (\ref{eq9}), we have 
$$
\sum_{j \in G_n(\theta,x)} r_j(\theta,x)  \leq \sum_{j\in G_{n}(\tau,y)} r_j(\tau,y)+ 4(s-1) \leq \sum_{j\in G_{n}(\tau,y)} r_j(\tau,y) + \dfrac{cn}{D-1}
$$
for $\alpha$ sufficiently small such that $4(D-1)\left[C_0\log(1/\alpha)\right]^{-1} \leq c$.

Then, it follows from estimate (\ref{eq8}) that $$\log\prod_{j=1}^{n-1} |\partial_{x}f(\theta_j,x_j)| \geq (D+1)cn - (D-1)\sum_{j\in G_n(\tau,y)} r_j(\tau,y)$$
and the result follows.
\end{proof}

Let $J(r) = \{x \colon |x - \tilde{x}| < \sqrt[D]{\alpha} e^{-r} \}$.
We define
$$
B_2(n) = \left\{ (\theta, x) \colon \text{there is } 1\le j < n \text{ with } x_j \in I(\lfloor\sqrt[D]{n} \rfloor) \right\}.
$$
Using estimates from \cite[Section~4]{HMS24}, we obtain
$$
m(B_2(n)) \le \text{const } e^{-\sqrt{n}/4}.
$$

Additionally, consider the set 
$$
B_1(n) = \left\{ (\theta,x) \notin B_2(n) \colon \sum_{i\in G_n} r_{i}(\theta,x) \ge cn \right\}.
$$
From estimates in \cite[Section~4]{HMS24}, it follows that 
$$
m(B_1(n)) \le \text{const } e^{-\gamma n},
$$
for some constant $\gamma > 0$.

We define $E_n = B_1(n) \cup B_2(n)$. 
Then
 \begin{itemize}
  \item[(i)] $m(E_n)\leq e^{-\gamma\sqrt{n}}$, for some constant $\gamma > 0$; 
  \item[(ii)] If $\theta\in \mathbb{S}^1 \times \mathcal{M }\setminus E_n$ then $\displaystyle\sum_{i\in G_n(\theta,x)} r_i(\theta,x) \leq c_1n$.  
\end{itemize}

Following reasoning in \cite[Section~4]{HMS24}, we conclude that
\begin{equation}
\label{eq.nue}
\sum_{i=0}^{n-1} \log || D\varphi(\theta_i,x_i)^{-1} || \le -cn, \text{ for every } (\theta,x) \notin B_1(n) \cup B_2(n),
\end{equation}
see also \cite[Section~6.2]{AA03}.
Thus, the map $\varphi$ is nonuniformly expanding.

Let $d_\delta((\theta,x),\mathcal{C})$ denote the $\delta$-\emph{truncated distance} from $(\theta,x)$ to the critical set $\mathcal{C} = \{(\theta, x) \colon x = \tilde{x}\}$ defined as $d_\delta((\theta,x),\mathcal{C}) = d((\theta, x),\mathcal{C})$ if $d((\theta,x),\mathcal{C}) \le \delta$ and $d_\delta((\theta,x),\mathcal{C}) = 1$ otherwise. 
For $\delta = \dfrac{1}{D-1}\left(\dfrac{1}{D} -\dfrac{2\eta}{D-1} \right)\log\left(\dfrac{1}{\alpha}\right)$, considering the definition of $I(r)$ and $r_i$, we obtain the bound
\begin{equation}
\label{eq.lower appr}
\sum_{i=0}^{n-1} -\log d_\delta(\varphi^i(\theta, x), \mathcal{C}) \le \gamma n \quad \text{for } (\theta,x) \notin B_1(n) \cup B_2(n). 
\end{equation}
Thus, the orbits of $\varphi$ exhibit \emph{slow approximation} to the critical set $\mathcal{C}$.

\begin{remark}
\label{rmk23} 
From the definition of the map $\varphi$, we can conclude that these maps \textit{behave like a power of the distance} to the critical set $\mathcal{C}$: there exist constants $B>1$ and $\beta >0$ such that for every $(\theta,x) \in \mathcal{S}^1 \times \mathcal{M} \setminus \mathcal{C}$, we have
$$
\frac{1}{B} \; d(x,\mathcal{C})^\beta \le \frac{||D\varphi (\theta,x)v||}{||v||}  \le B\; d(x,\mathcal{C})^{-\beta}, \text{ for every } v \in T_{(\theta,x)} \mathcal{S}^1 \times \mathcal{M}.
$$
Furthermore, from equations \eqref{eq.nue} and \eqref{eq.lower appr} it follows that the set $\Gamma_n$ consisting of points $(\theta,x)$ that exhibit nonuniformly expansion and slow approximation to the critical set, satisfies:
$$
m(\Gamma_n) \le \mathcal{O}(n^{-\zeta}), \text{ for every } \zeta > 0.
$$
As a result, we conclude that $\varphi \in \mathcal{N}$ satisfies the hypothesis in \cite[Theorem~2]{ALP05}
\end{remark}

\medskip

\begin{defi}
	Given $0<\tilde{\sigma}<1$, we say $n$ is a $\tilde{\sigma}$\emph{-hyperbolic time} for $(\theta,x)\in\mathbb{S}^1\times \mathcal{M}$ if $$\prod_{i=k}^{n-1}\left\| D\varphi(\theta_i,x_i)^{-1} \right\| \leq \tilde{\sigma}^{n-k},$$ for $k=0,\ldots, n-1$.
	\label{def23}
\end{defi} 

\begin{remark}
Under a mild assumption on the derivative, Definition \ref{def23} implies $$
\sum_{i\in G_n(\theta,x), i\geq k} r_i(\theta,x) \leq \dfrac{1}{D-1}	(c+\varepsilon)(n-k),
$$ 
for $k=0,\ldots, n-1$.
\end{remark}

\noindent Indeed, fixing  $0< \varepsilon \leq c/2$. 
Since $d\geq 16$, it follows that $e^{(D+1)c+\varepsilon} \leq d-\alpha$, for $\alpha$ sufficiently small. 
By taking the norm $\left\| D\varphi \right\|$ to be the maximum norm of its entries, a simple calculation together with estimates on the derivatives of $g$ and $f$ shows that $\left\|D\varphi(\theta,x)^{-1}\right\| = \left|\partial_x f(\theta,x)^{-1}\right|$. Then, from \eqref{eq8} and the definition of $\tilde{\sigma}$-hyperbolic time, we will assume that $$(D+2) c(n-k) - (D-1)\sum_{i\in G_n, i\geq k} r_{i}(\theta,x)  \geq  (n-k)\log\left(\tilde{\sigma}^{-1}\right).$$ Now, we can take $\tilde{\sigma}^{-1} = e^{(D+1)c-\varepsilon}$ to obtain $$ (D+2) c(n-k) - (D-1)\sum_{i\in G_n, i \geq k} r_{i}(\theta,x)  \geq [(D+1)c-\varepsilon](n-k)$$ and the claim follows.

We say that the hyperbolic time $n$ is a \emph{hyperbolic return} if $n$ is also a return situation for $(\theta,x)$.

\medskip

Fix an integer $p\geq 0$ sufficiently large. Define the sets $H_n\subset \mathbb{S}^1\times \mathcal{M}$ consisting of points in $\mathbb{S}^1\times \mathcal{M}$ whose first hyperbolic return greater than $p$ is $n$ and let $$H = \bigcup_{n\geq p} H_n.$$
We also have the sets $H^{*}_n \subset \mathbb{S}^1\times \mathcal{M}$ of those points $(\theta,x)$ whose first hyperbolic time is $n$ and $H^{*} = \bigcup_{n\geq p} H^{*}_n$. Clearly, $H_n\subset H^{*}_n$, for all $n\geq p$.

It follows from a lemma by Pliss (see \cite[Lemma 11.8]{Man87}) that $H^{*}$ has full Lebesgue measure.

\begin{prop}
There is an integer $n_0 = n_0(p,\varepsilon) \geq p$ such that for every $n\geq n_0$ we have 
$$
(\mathbb{S}^1 \times \mathcal{M}) \setminus E_n \subset H^{*}_p\cup \ldots \cup H^{*}_n.
$$
\label{hypetimesprop}
\end{prop}

\begin{proof}
 See Proposition 2.6 in \cite{Al00}. 
\end{proof}

By \cite[Lemma 4.4]{AV02}, the set $H$ of hyperbolic returns also have full Lebesgue measure.

\section{Partition of $\mathbb{S}^1\times \mathcal{M}$}
We begin defining a partition of $ \mathcal{M}$ by completing the intervals $I_r$ introduced in the previous section which we recall below: 
$$
I_r = \left(\tilde{x} + \sqrt[D]{\alpha} e^{-r} ,    \tilde{x} + \sqrt[D]{\alpha}e^{-r+1}\right] \ \mbox{for} \ r\geq 1, 
$$ 
$$
 I_r = \left[\tilde{x} - \sqrt[D]{\alpha} e^{r+1} , \tilde{x} - \sqrt[D]{\alpha}e^{r}\right), \ \mbox{for} \ r\leq -1.
$$
Then we write
$$
I_{0^{+}} = \left(\tilde{x} + \sqrt[D]{\alpha}, \tilde{x} + e\sqrt[D]{\alpha}\right] \ \ \mbox{and} \ \  I_{0^{-}} = \left[\tilde{x} - \sqrt[D]{\alpha}, \tilde{x} - e\sqrt[D]{\alpha}\right).
$$
For the case $\mathcal{M} = I$, we put 
$$
I_{+} = \left(I\setminus [\tilde{x}-e\sqrt[D]{\alpha},\tilde{x}+e\sqrt[D]{\alpha}]\right)\cap \mathbb{R}^{+} \ \mbox{and}  \ 
I_{-} = \left(I\setminus [\tilde{x}-e\sqrt[D]{\alpha},\tilde{x}+e\sqrt[D]{\alpha}]\right)\cap \mathbb{R}^{-}
$$
and when $\mathcal{M} = \mathbb{S}^1$, we take 
$$
I^{c} = \mathbb{S}^1\setminus [\tilde{x}-e\sqrt[D]{\alpha},\tilde{x}+e\sqrt[D]{\alpha}].
$$
This partition of $\mathcal{M}$ induces partitions on each fiber $\left\{\theta\right\}\times \mathcal{M} \subset \mathbb{S}^1\times \mathcal{M}$, for which we slightly abuse the notation and also refer to them as $I_r$, $I_{0^{+}}$, $I_{0^{-}}$ and $I^c$ for each $\theta \in \mathbb{S}^1$, when they make sense according to our definitions above.

We also introduce the following notation: 
$$
I_r^{+} = I_{r-1} \cup I_r\cup I_{r+1}, \ \ \mbox{for} \ \ |r| \geq 1,
$$ 
$$
I_{0^{+}}^{+} = I_{+}\cup I_{0^{+}} \cup I_{1} \ \  \mbox{and} \ \  I_{0^{-}}^{+} = I_{-}\cup I_{0^{-}} \cup I_{-1},
$$ 
if $\mathcal{M} = I$, and  
$$
I_{0^{+}}^{+} = I^c\cup I_{0^{+}} \cup I_{1} \ \ \mbox{and} \ \ I_{0^{-}}^{+} = I^c\cup I_{0^{-}} \cup I_{-1},
$$ 
if $\mathcal{M} = \mathbb{S}^1$.

We will now use the sets $H_n$ together with other requirements to construct a partition $\mathcal{R}$ of $\mathbb{S}^1 \times \mathcal{M}$ by rectangles as in \cite{Al00} to create a piecewise uniformly expanding map and apply his results about those kinds of maps in the construction of the invariant measure.

To construct this partition we consider initially the partition $\mathcal{Q}$ of the interval $I$ by the subintervals $I_r$, $I^{+}$ and $I^{-}$ and the following Markov partition of $\mathbb{S}^1$:
\begin{itemize}
  \item[(i)] $\mathcal{P}_1 =\left\{[\theta_{i-1},\theta_{i}); \ i = 1,\ldots, d \right\}$, where $\theta_0$ is the fixed point of $g$ closest to $0$ and  $\theta_0,\theta_1,\ldots, \theta_d = \theta_0$ are the preimages of $\theta_0$ under $g$, ordered according to the orientation of $\mathbb{S}^1$;
  \item[(ii)] $\mathcal{P}_n = \left\{\mbox{connected components of} \ g^{-1}(\omega); \ \omega \in \mathcal{P}_{n-1} \right\}$, for $n\geq 2$.    
\end{itemize}
Also, given $\omega \in \mathcal{P}_n$, denote by $\omega^{-}$ the left hand endpoint of $\omega$.

Then we construct $\mathcal{R} = \bigcup_{n\geq p}\mathcal{R}_n$ inductively, starting with the partition $\mathcal{P}_p\times\mathcal{Q}$, subdividing its rectangles and creating sets $\mathcal{R}_n$ of these subdivided rectangles at each step $n \geq p$, satisfying certain properties that we explain now (for more details, see \cite[Section 3]{Al00}).

\subsection{Requirements for the elements of $\mathcal{R}_n$}
The idea is to create these partitions in such a way that the restriction of certain iterations of $\varphi$ to the interior of these rectangles are uniformly expanding and $C^2$-diffeos onto its images. 
To guarantee that, we need four conditions to hold for the rectangles of $\mathcal{R}_n$, $n\geq p$:
\begin{itemize}
\item[] $(I_n)$ $H_n\subset \bigcup_{R\in\mathcal{R}_n} R$ and $R\cap H_n \neq \emptyset$ for every $R\in\mathcal{R}_n$. 
\item[] $(II_n)$ For every $0\leq j \leq n$ and $\omega \times J\in \mathcal{R}_n$ there is $I_{r_j}\in \mathcal{Q}$ such that $\varphi^j(\left\{\omega^{-}\right\}\times J) \subset I_{r_j}^{+}$, where $I_{r_j}^{+} = I_{r_j+1}\cup I_{r_j}\cup I_{r_j-1}$.
\end{itemize}

To state the other conditions, we consider the following subset of $\mathcal{R}_n$:
$$
\mathcal{R}_n^{*} = \left\{\omega\times J\in \mathcal{R}_n \ | \ \exists \ 0\leq j < n, \ \exists \ I_{r_j}\in \mathcal{Q}: \ I_{r_j}\subset \varphi^j(\left\{\omega^{-}\right\}\times J) \right\}.
$$

\begin{defi}
We will say that $\omega_n\times J_n \in\mathcal{R}_n$ is \emph{subordinate} to $\omega_l\times J_l \in \mathcal{R}_{l}^{*}$ if $\omega_n\subset \omega_l$, $J_n$ and $J_l$ have a common endpoint, and there is $j\leq l$ and $I_{r_j}\in \mathcal{Q}$ for which the following holds:
\begin{itemize}
\item[(i)] $I_{r_j} \subset \varphi^j(\left\{\omega^{-}\right\}\times J_l)$; 
\item[(ii)] $I_{r_j+1}\subset \varphi^j(\left\{\omega_l^{-}\right\}\times J_n)$ or $I_{r_j-1} \subset \varphi^j(\left\{\omega_l^{-}\right\}\times J_n)$. \end{itemize}
\end{defi} 

\vspace{0.2cm}

The third condition required on the rectangles is
\begin{itemize}
\item[] $(III_n)$ For every $R\in\mathcal{R}_n$, either $R\in\mathcal{R}_n^{*}$ or $R$ is subordinate to some $R^{*}\in \mathcal{R}_l^{*}$ with $l\leq n$.
\end{itemize}

This condition guarantees that the rectangles in $\mathcal{R}_n$  eventually have \emph{large size}, which is required to prove the existence of invariant measures.

At each step $n\geq p$ the inductive process will create another partition $\mathcal{S}_n$ that contains the set of points that are not in the rectangles $R\in\mathcal{R}$ constructed at moment $n$, that is, $\mathcal{S}_n$ is the partition of the set $$(\mathbb{S}^1\times \mathcal{M}) \setminus \bigcup_{i=p}^n \bigcup_{R\in\mathcal{R}_i} R.$$  

Rectangles in $\mathcal{S}_n$ will also have the form $\omega\times J$, with $\omega \in \mathcal{P}_n$ and $J$ is a subinterval of some interval $I_{r_j}\in\mathcal{Q}$. The rectangles $R\in \mathcal{R}_{n+1}$ are constructed out of those rectangles in $\mathcal{S}_n$ and so, to ensure property $(III)_j$ for rectangles in $\mathcal{R}_j$ with $j>n$, we will require that for all $n\geq p$ the following holds:

\begin{itemize}
\item[]  $(IV_n)$ For every $\omega\times J\in \mathcal{S}_n$, either $J = I_r$, for some $I_r\in\mathcal{Q}$, or $\omega\times J$ is subordinate to some $R^{*}\in\mathcal{R}_l^{*}$ with $l \leq n$.
\end{itemize}

\subsection{Construction of the partition}

The construction is done inductively. For the first step, take an arbitrary $\omega_p\in\mathcal{P}_p$ and let $\mathcal{J}_0$ be the family of intervals $I_r\in\mathcal{Q}$ such that $(\omega_p\times I_r)\cap H_p\neq \emptyset$. Now take the sets $\varphi(\left\{\omega_p^{-}\right\}\times J_0)$, with $J_0\in\mathcal{J}_0$, and consider the following two possible cases:

\vspace{0.3cm}

\hspace{0.2cm} (a) $I_r\subset \varphi(\left\{\omega_p^{-}\right\}\times J_0)$, for some $I_r\in\mathcal{Q}$.

\vspace{0.2cm} 

In this case we write $J_0 = \bigcup_{i_1} J_{i_1}$, where the $J_{i_1}$'s are intervals satisfying 
\begin{equation}
I_{r_{i_1}}\subset \varphi(\left\{\omega_p^{-}\right\}\times J_{i_1}) \subset I_{r_{i_1}}^{+},
\label{inc1}
\end{equation}
for some $I_{r_{i_1}} \in\mathcal{Q}$. We obtain the intervals $J_{i_1}$ by taking $$J_{i_1} = J_0\cap \varphi^{-1}(\left\{g(\omega_p^{-})\right\}\times I_{r_{i_1}}),$$ except for the two end subintervals in $J_0$, which, if necessary, may be joined to the adjacent ones in order to guarantee the first inclusion in (\ref{inc1}).

We take $\mathcal{J}_1$ to be the sets $J_{i_1}$ in the union above such that $(\omega_p\times J_{i_1})\cap H_p \neq \emptyset$.

\vspace{0.3cm}

\hspace{0.2cm} (b) $\varphi(\left\{\omega_p^{-}\right\}\times J_0)$ does not contain any $I_r\in\mathcal{Q}$.

\vspace{0.2cm}

In this case, we do not divide $J_0$ and say $J_0\in\mathcal{J}_1$.

\vspace{0.3cm}

Now we take $J_1\in\mathcal{J}_1$ and consider the sets $\varphi^2(\left\{\omega_p^{-}\right\}\times J_1)$. If $I_r\subset\varphi^2(\left\{\omega_p^{-}\right\}\times J_1)$, for some $I_r\in\mathcal{Q}$, we decompose $J_1 = \bigcup_{i_2} J_{i_2}$ as above and take $\mathcal{J}_2$ to be the family of those intervals $J_{i_2}$ satisfying $(\omega_p\times J_{i_2})\cap H_p \neq \emptyset$. On the other hand, if $\varphi^2(\left\{\omega_p^{-}\right\}\times J_1)$ does not contain any $I_r\in \mathcal{Q}$, we say $J_1 \in\mathcal{J}_2$.

We procede like that until the $(p-1)$th iterate, defining in this way the family of sets $\mathcal{J}_{p-1}$. Let $\mathcal{C}_{p-1}$ be the of connected components of $$J_0 \setminus \bigcup_{J\in \mathcal{J}_{p-1}} J.$$

Then, given $J \in \mathcal{J}_{p-1}\cup\mathcal{C}_{p-1}$, we say that $\omega_p\times J\in \mathcal{R}_p$ if $J\in\mathcal{J}_{p-1}$, and $\omega_p\times J \in\mathcal{S}_p$ if $J\in\mathcal{C}_{p-1}$. Repeating this procedure with all $\omega_p\in \mathcal{P}_p$ we obtain all the rectangles in $\mathcal{R}_p$ and $\mathcal{S}_p$. Constructed in this way, the have collections $\mathcal{R}_p$ and $\mathcal{S}_p$ satisfies conditions ($I_p$)-($IV_p$), and in fact we have $\mathcal{R}_p = \mathcal{R}_p^{*}$.

Now suppose we have defined families $\mathcal{R}_p,\ldots, \mathcal{R}_n$ and $\mathcal{S}_n$ satisfying ($I_n$)-($IV_n$). We define $\mathcal{R}_{n+1}$ and $\mathcal{S}_{n+1}$ inductively as follows.
Take $S \in\mathcal{S}_n$. By inductive hypothesis, it follows that $S = \omega_{n}\times J_n$ with $\omega_n\in\mathcal{P}_n$ and $J_n \subset I_r$, for some $I_r\in\mathcal{Q}$. We write $$S = \bigcup_{i=1}^{d} \left(\omega^i_{n+1}\times J_n \right),$$ where $\omega_{n+1}^1,\ldots,\omega^d_{n+1}$ are intervals in $\mathcal{P}_{n+1}$ that cover $\omega_n$. We then distinguish the following two cases:

\vspace{0.2cm}

\hspace{0.2cm} (a) $\left(\omega^i_{n+1}\times J_n\right)\cap H_n = \emptyset$.
 
\vspace{0.2cm}

In this case we say that $\omega^i_{n+1}\times J_n\in \mathcal{S}_{n+1}$ and it's obvious that property ($IV_{n+1}$) is true since $J_n$ has not been divided.

\vspace{0.2cm}

\hspace{0.2cm} (b) $\left(\omega^i_{n+1}\times J_n\right)\cap H_n \neq \emptyset$.

\vspace{0.2cm}

In this case we also have two possible cases:

\vspace{0.1cm}

\hspace{0.2cm} (i) $\exists \ 0\leq j \leq n$ and $\exists \ I_{r_j}\in \mathcal{Q}$ such that $I_{r_j} \subset \varphi^j(\omega^{i-}_{n+1}\times J_n)$.

\vspace{0.2cm}

In this case, we make divisions of $J_n$ as we in the first step, starting the process with $\omega^i_{n+1}\times J_n$ instead of $\omega\times J_0$, defining in this way rectangles in $\mathcal{R}_{n+1}$ and $\mathcal{S}_{n+1}$ contained in $\omega^{i}_{n+1}\times J_{n}$. As before, conditions ($I_{n+1}$)-($IV_{n+1}$) are verified directly by the construction.

\vspace{0.2cm}

\hspace{0.2cm} (ii) $\varphi^j(\omega^i_{n+1}\times J_n)$ contains no $I_r\in\mathcal{Q}$, for $0\leq j \leq n$.

\vspace{0.2cm}

In this case we say that $\omega^i_{n+1}\times J_n\in \mathcal{R}_{n+1}$ and this implies that ($II_{n+1}$) is true. Indeed, for each $0\leq j \leq n$, there must be some $I_{r_j}\in\mathcal{Q}$ such that $\varphi^j(\omega^{i-}_{n+1}\times J_n)\cap I_{r_j} \neq \emptyset$ and so $\varphi^j(\omega^{i-}_{n+1}\times J_n) \subset I_{r_j}^{+}$, otherwise, either $I_{r_j-1}$ or $I_{r_j+1}$ would be contained in $\varphi^j(\omega^{i-}_{n+1}\times J_n)$. Condition ($III_{n+1}$) is also true since no division was made in $J_n$ and ($IV_{n+1}$) also follows from the construction.  

\vspace{0.2cm}

The induction is complete. Since ($I_n$) is valid for all $n \geq p$ and $H = \bigcup_{n\geq p} H_n$ has full Lebesgue measure, it follows that $\mathcal{R} = \bigcup_{n\geq p} \mathcal{R}_n$ is indeed a partition of $\mathbb{S}^1\times I$. 

We now prove some geometrical properties of the rectangles in the partition $\mathcal{R}$ that will be required later. A set $\hat{X}\subset \mathbb{S}^1\times \mathcal{M}$ is an \emph{admissible curve} if it is the graph of a map $X : \mathbb{S}^1 \rightarrow \mathcal{M}$ satisfying the following conditions:

\begin{itemize}
	\item $X$ is $C^2$ except, possibly, being discontinuous on the left at $\theta = \theta_0$;
	\item $|X'(\theta)|\leq \alpha$ and $|X''(\theta)|\leq \alpha$ at every $\theta \in \mathbb{S}^1$. 
\end{itemize}

Given $\omega\subset \mathbb{S}^1$, we denote $\hat{X} | \omega = G(X|\omega)$, where $G(X|\omega)$ is the graph of the restriction of the map $X$ to the subset $\omega$.

\begin{lema}

If $\hat{X}$ is an admissible curve and $\omega \in \mathcal{P}_n$, then $\varphi^n(\hat{X}|\omega)$ is also an admissible curve.
\label{lemma32}
\end{lema} 

\begin{proof}
See \cite[Lemma 2.1]{HMS24}, .
\end{proof}

\begin{cor}
If $R\in\mathcal{R}_n$ for some $n\geq p$, then the boundary of $\varphi^n(R)$ is made of two vertical lines and two admissible curves.
\label{cor33}
\end{cor}

\begin{proof}
Follows immediately from the construction of the rectangles and Lemma \ref{lemma32}.
\end{proof}

For the next lemma, we suppose that $0 < \eta \leq \frac{1}{4}$.

\begin{lema}
	There is some constant $\delta_0>0$ such that if $\alpha$ is sufficiently small, then for every $(\sigma,y)\in H_n$ and $0\leq j \leq n$ we have
	$$ |I_{r_j(\sigma,y)+5}| \geq  \delta_0 \cdot \alpha^{\frac{1}{D-1} \left(1-\frac{2\eta}{D-1}\right)} \cdot e^{-(c+\varepsilon)(n-j)}\geq 4\alpha(d-\alpha)^{-(n-j)}.$$
	\label{lemma34}
\end{lema}

\begin{proof}
We split the proof in two cases.

\vspace{0.2cm}

\noindent {\bfseries Case 1:} $j\in G_n(\sigma,y)$\\
In this case we have, in particular, $r_j(\sigma,y) \leq \dfrac{1}{D-1}(c+\varepsilon)(n-j)$ and then 
\begin{align*}
|I_{r_j(\sigma,y)+5}| & = \alpha^{\frac{1}{D}}e^{-(r_j(\sigma,y)+4)} - \alpha^{\frac{1}{D}}e^{-(r_j(\sigma,y)+5)} 
  =\alpha^{\frac{1}{D}}e^{-r_j(\sigma,y)}\left(e^{-4}-e^{-5}\right)\\
  & \geq \alpha^{\frac{1}{D-1} \left(1-\frac{2\eta}{D-1}\right)}\cdot e^{-(c+\varepsilon)(n-j)}\cdot \left(e^{-4}-e^{-5} \right),
\end{align*}
for $D\geq 2$ and $\alpha$ sufficiently small.

\vspace{0.2cm}

\noindent {\bfseries Case 2:} $j\notin G_{n}(\sigma,y)$\\
In this case, we have $$r_j(\sigma,y) \leq \dfrac{1}{D-1}\left(\dfrac{1}{D} - \dfrac{2\eta}{D-1} \right)\log\left(
\dfrac{1}{\alpha}\right),$$ and then 
\begin{align*}
|I_{r_j(\sigma,y)+5}|  & = \alpha^{\frac{1}{D}}e^{-(r_j(\sigma,y)+4)} - \alpha^{\frac{1}{D}}e^{-(r_j(\sigma,y)+5)} 
 = \alpha^{\frac{1}{D}}e^{-r_j(\sigma,y)}\left(e^{-4}-e^{-5}\right) \\
&  \geq \alpha^{\frac{1}{D-1} \left(1-\frac{2\eta}{D-1}\right)}\cdot \left(e^{-4}-e^{-5}\right) 
\end{align*}
In any case, we take $\delta_0 = e^{-4}-e^{-5}$.

The second inequality follows from the assumptions made on $d$, $\varepsilon$ and $\alpha$, where $\alpha$ is chosen small enough such that $\delta_0 \alpha^{\frac{1}{D-1}\left(1-\frac{2\eta}{D-1}\right)} \geq 4\alpha$.
\end{proof}

\begin{lema}
	Let $n\geq p$ and $R\in\mathcal{R}_n$. If $(\theta,x)\in R$ and $(\sigma,y)\in R\cap H_n$ then $r_j(\theta,x)\leq r_j(\sigma,y)+4$, for all $0\leq j \leq n-1$.
	\label{lemma35}
\end{lema}

\begin{proof}

Take $0\leq j < n$ and define $(\theta_j,x_j) = \varphi^j(\theta,x)$ and $(\omega_j^{-},x_j^{-}) = \varphi^j(\omega^{-},x)$. Since $(\theta_j,x_j)$ and $(\omega_j^{-},x_j^{-})$ lies in the same admissible curve, it follows that $$|\theta_j - \omega_j^{-}| \leq (d-\alpha)^{-(n-j)}.$$ By Lemma \ref{lemma32} and the mean-value theorem we also have $$|x_j^{-}-x_j| \leq \alpha (d-\alpha)^{-(n-j)},$$ which implies $$|x_j|\geq |x_j^{-}|-\alpha (d-\alpha)^{-(n-j)} \geq \alpha^{\frac{1}{D}} e^{-r_j(\omega^{-},x)} - \alpha(d-\alpha)^{-(n-j)}.$$

Now, taking $(\sigma_j,y_j) = \varphi^j(\sigma,y)$ and $(\omega_j^{-},y_j^{-}) = \varphi^j(\omega^{-},y)$, the argument above also applies and $$|y_j^{-}| \geq |y_j| - \alpha(d-\alpha)^{-(n-j)} \geq \alpha^{\frac{1}{D}} e^{-r_j(\sigma,y)} - \alpha(d-\alpha)^{-(n-j)}.$$

By Lemma \ref{lemma34}, we also have $$|I_{r_j(\sigma,y)+1}|\geq |I_{r_j(\sigma,y)+5}| \geq 4\alpha(d-\alpha)^{-(n-j)} \geq \alpha(d-\alpha)^{-(n-j)}$$ and then 
$$ |y^{-}_j| \geq \alpha^{\frac{1}{D}} e^{-r_{j}(\sigma,y)} - |I_{r_{j}(\sigma,y)+1}|  = \alpha^{\frac{1}{D}} e^{-(r_j(\sigma,y)+1)},$$ which implies that 

\begin{equation}
r_j(\omega^{-},y) \leq r_j(\sigma,y) +1.
\label{eq1}
\end{equation}

By property ($II_n$) of rectangles in $\mathcal{R}_n$, we have $x^{-}_j \in I_{r_j(\omega^{-},y)}^{+}$ and so 

\begin{equation}
r_j(\omega^{-},x) \leq r_j(\omega^{-},y)+2.
\label{eq2}
\end{equation}

Combining (\ref{eq1}) and (\ref{eq2}) we get 

\begin{equation}
r_{j}(\omega^{-},x) \leq r_j(\sigma,y)+3.
\label{eq3}
\end{equation}

Finally, since $|I_{r_j(\sigma,y)+4}| \geq |I_{r_j(\sigma,y)+5}|$ we can apply Lemma \ref{lemma34} again together with (\ref{eq3}) to obtain 
\begin{equation*}
\begin{split}
|x_j| & \geq \alpha^{\frac{1}{D}} e^{-r_j(\omega^{-},x)} - \alpha(d-\alpha)^{-(n-j)} 
 \geq \alpha^{\frac{1}{D}} e^{-(r_j(\sigma,y)+3)} - |I_{r_{j}(\sigma,y)+4}| 
 = \alpha^{\frac{1}{D}} e^{-r_j(\sigma,y)+4}, 
\end{split}
\end{equation*}
which implies $r_j(\theta,x) \leq r_j(\sigma,y)+4$.

\end{proof}

\begin{cor}
For any $n\geq p$ and $R\in\mathcal{R}_n$, the map $\varphi^n|_R$ is a diffeomorphism onto its image.
\label{cor36}
\end{cor}

\begin{proof}
Let $n\geq p$ and $R\in\mathcal{R}_n$. By property ($I_n$), $R\cap H_n \neq \emptyset$. The points in $R\cap H_n$ does not hit the critical line $\left\{x=0 \right\}$ in the first $n-1$ iterates and by Lemma \ref{lemma35} the same follows for any point in $R$. This implies $\varphi^n|_R$ is a diffeomorphism onto its image.

\end{proof}

\begin{lema}
Let $n\geq p$ and $R\in\mathcal{R}_n$. If $(\theta,x)\in R$, then for $j = 0,\ldots, n-1$ we have $$\prod_{i=j}^{n-1} |\partial_xf(\theta_i,x_i)| \geq \exp \left( (Dc-\varepsilon)(n-j) \right).$$  
\label{lemma37}
\end{lema}

\begin{proof}
Take $n\geq p$ and $R\in\mathcal{R}_n$. Then, by property ($I_n$) there is some $(\sigma,y)\in R\cap H_n$. By Lemma \ref{lemma35}, we have $r_i(\theta,x) \leq r_i(\sigma,y)+4$, for $i= 0,\ldots, n-1$. This, in particular, implies that $$r_i(\theta_j,x_j) \leq r_i(\sigma_j,y_j)+4,$$ for $i = 0,\ldots, n-j-1$. Applying Lemma \ref{lemma23}, we get $$\prod_{i=0}^{n-j-1} |\partial_x f(\theta_{j+i}, x_{j+i})|  \geq \exp\left( (D+1)c(n-j) - (D-1)\sum_{i\in G_{n-j}(\sigma_j,y_j)} r_i(\sigma_j,y_j) \right).$$

Since $n$ is a hyperbolic time for $(\sigma,y)$, it follows that $n-j$ is a hyperbolic time for $(\sigma_j,y_j)$, and then   
$$\prod_{i=0}^{n-j-1} |\partial_x f(\theta_{j+i}, x_{j+i})|  \geq \exp\left( (D+1)c(n-j) - (c+\varepsilon)(n-j) \right),$$ which gives $$\prod_{i=j}^{n-1} |\partial_xf(\theta_i,x_i)| \geq \exp\left( (Dc-\varepsilon)(n-j) \right).$$

\end{proof}

\begin{prop}
There is $\delta_1 = \delta_1(\alpha)>0$ such that for each $n\geq p$ and $\omega\times J \in \mathcal{R}_n$ we have $|\varphi^n(\left\{\theta\right\}\times J)| \geq \delta_1$, for every $\theta\in\omega$.
\label{prop38}
\end{prop}

\begin{proof}
Let $\omega\times J \in\mathcal{R}_n$ be any rectangle and fix $\theta \in \omega$. There are two possible cases here: either $\omega\times H \in\mathcal{R}_n^{*}$ or $\omega\times J \notin \mathcal{R}_n^{*}$. We split the proof in these two cases.

\vspace{0.2cm}

\noindent {\bfseries{Case 1}}: $\omega\times J \in\mathcal{R}_n^{*}$.\\
In this case, we know there is some $0\leq j \leq n-1$ and $I_{r_j}\in\mathcal{Q}$ such that 
\begin{equation}
I_{r_j}\subset \varphi^j(\left\{\omega^{-} \right\}\times J).
\label{eq14}
\end{equation}

By the mean-value theorem, there is some $x\in J$ such that 
\begin{equation*}
\left|\varphi^n(\left\{\theta\right\} \times J)\right| = \prod_{i=j}^{n-1} \left|\partial_xf(\theta_i,x_i)\right|\cdot \left|\varphi^j(\left\{\theta\right\}\times J)\right|
\end{equation*}
and, by Lemma \ref{lemma37}, we get 
\begin{equation}
\left|\varphi^n(\left\{\theta\right\} \times J)\right| \geq \exp\left(
(Dc-\varepsilon)(n-j)\right) \cdot \left|\varphi^j(\left\{\theta\right\}\times J)\right|.
\label{eq15}
\end{equation}

Now it suffices to prove that $|\varphi^j(\left\{\theta\right\}\times J)|$ has a lower bound. Let $J = [u,v]$ and consider the two curves $\gamma_1 = \varphi^j(\omega\times\left\{u\right\})$ and $\gamma_1 = \varphi^j(\omega\times\left\{v\right\})$. 

By Lemma \ref{lemma32}, $\gamma_1$ and $\gamma_2$ are contained in admissible curves and so they are images of maps defined on $g^j(\omega) \in \mathcal{P}_{n-j}$ whose derivatives have absolute value bounded above by $\alpha$. Applying the mean-value theorem to these maps shows that the diameter of these curves in the $x$-direction are bounded above by $\alpha(d-\alpha)^{-(n-j)}$. Using this fact together with (\ref{eq14}) and assuming $|u_j|<|v_j|$ (the other case is similar) gives the following estimates to the points $(\theta_j,u_j) = \varphi^j(\theta,u)$ and $(\theta_j,v_j) = \varphi^j(\theta, v)$: 
\begin{equation}
 |u_j| \leq \alpha^{\frac{1}{D}} e^{-r_j} + \alpha(d-\alpha)^{-(n-j)}
 \label{eq16}
\end{equation}
and
\begin{equation}
|v_j|\geq \alpha^{\frac{1}{D}} e^{-(r_j-1)} - \alpha(d-\alpha)^{-(n-j)} .
\label{eq17}
\end{equation}

Take $(\sigma,y)\in(\omega\times J)\cap H_n$. Since $r_j = r_j(\omega^{-},z)$, for some $(\omega^{-},z) \in\omega \times J$, it follows from Lemma \ref{lemma35} that $r_j \leq r(\sigma,y)+4$. By Lemma \ref{lemma34} we have
\begin{equation}
|I_{r_j}| > |I_{r_j(\sigma,y)+4}| > |I_{r_j(\sigma,y)+5}| \geq 4\alpha(d-\alpha)^{-(n-j)}. 
\label{eq18} 
\end{equation} 

Then, by (\ref{eq16}), (\ref{eq17}) and (\ref{eq18}), we get 
\begin{equation*}
\begin{split}	
|v_j - u_j| & \geq \alpha^{\frac{1}{D}}e^{-(r_j-1)} - \alpha^{\frac{1}{D}}e^{-r_j} - 2\alpha (d-\alpha)^{-(n-j)}
 = |I_{r_j}| - 2\alpha(d-\alpha)^{-(n-j)} > \dfrac{|I_{r_j(\sigma,y)+4}|}{2}.
\end{split}
\end{equation*}

Hence, by Lemma \ref{lemma34}, we have

$$|\varphi^j(\left\{\theta\right\}\times J)| \geq \dfrac{\delta_0}{2} \cdot \alpha^{\frac{1}{D-1} \left(1-\frac{2\eta}{D-1}\right)} e^{-(c+\varepsilon)(n-j)}.
$$  

Finally, plugging this last estimate into (\ref{eq15}) we obtain

\begin{equation}
|\varphi^n(\left\{\theta \right\}\times J)| \geq \dfrac{\delta_0}{2}\cdot \alpha^{\frac{1}{D-1} \left(1-\frac{2\eta}{D-1}\right)} \cdot \exp\left( ((D-1)c-2\varepsilon)(n-j) \right). 
\label{eq19}
\end{equation}
Since we choose $\varepsilon$ such that $c > 2\varepsilon$, which also implies $(D-1)c > 2\varepsilon$, the right side of (\ref{eq19}) clearly has a lower bound $\delta_1$ depending on $\alpha$, for all $n$ and $j\leq n$.

\vspace{0.2cm}

\noindent {\bfseries Case 2:} $\omega\times J \notin \mathcal{R}_n^{*}$.\\
In this case, by property ($III_n$) there is some $l\leq n$ and $\omega_l\times J_l\in\mathcal{R}_l^{*}$ such that $\omega\times J$ is subordinate to $\omega_l\times J_l$, i.e., there is $j< l$ and $I_{r_j}\in\mathcal{Q}$ such that $I_{r_j} \subset \varphi^j(\left\{\omega_{l}^{-}\right\}\times J_l)$ and  
\begin{equation*}
I_{r_j+1}\subset \varphi^j\left(\left\{\omega_l^{-}\right\}\times J\right) \ \ \ \mbox{or} \ \ \ I_{r_j-1}\subset \varphi^j\left(\left\{\omega_l^{-}\right\}\times J\right).
\end{equation*}

Suppose, without loss of generality, that $I_{r_j+1}\subset \varphi^j\left(\left\{\omega_l^{-}\right\}\times J\right)$. As in Case 1, we also have 
\begin{equation}
|\varphi^n(\left\{\theta\right\}\times J)| \geq \exp\left(
(Dc-\varepsilon)(n-j)\right) \cdot \left|\varphi^j(\left\{\theta\right\}\times J)\right|
\label{eq20}
\end{equation}
and again we take $J = [u,v]$ and consider the curves $$ \gamma_1 = \varphi^j(\omega_l\times \left\{u\right\}) \ \ \ \ \ \mbox{and} \ \ \ \ \ \gamma_2 = \varphi^j(\omega_l\times \left\{v\right\}).
$$
 
 As before these curves are contained in admissible curves defined on $g^j(\omega_l) \in\mathcal{P}_{l-j}$ and whose diameters in the $x$-direction are bounded above by  $\alpha(d-\alpha)^{-(l-j)}$. From a similar argument made in Case 1 we obtain the following estimates for $(\theta_j,u_j) = \varphi^j(\theta,u)$ and $(\theta_j,v_j)  = \varphi^j(\theta,v)$:
  \begin{equation}
 |u_j| \leq \alpha^{\frac{1}{D}}e^{-(r_j+1)} + \alpha(d-\alpha)^{-(l-j)}
 \label{eq21}
 \end{equation}
 and
 \begin{equation}
 |v_j|\geq \alpha^{\frac{1}{D}}e^{-r_{j}} - \alpha(d-\alpha)^{-(l-j)}.
 \label{eq22}
 \end{equation}
 
 Take $(\sigma,y)\in (\omega_l\times J_l)\cap H_l$, which exists by ($I_l$). As before, there is some $z\in J_l$ such that $r_j = r_j(\omega_l^{-},z)$ and from Lemma \ref{lemma35} we get $$r_j \leq r_j(\sigma,y) +4.$$
 From this and Lemma \ref{lemma34} it follows that  
 \begin{equation}
 |I_{r_j+1}| \geq |I_{r_j(\sigma,y)+5}| \geq 4\alpha(d-\alpha)^{-(l-j)}.
 \label{eq23}
\end{equation}
 By (\ref{eq21}), (\ref{eq22}) and (\ref{eq23}) we obtain
 \begin{equation*}
 \begin{split}
  |v_j-u_j| & \geq  \alpha^{\frac{1}{D}}e^{-r_j} - \alpha^{\frac{1}{D}}e^{r_j+1} - 2\alpha (d-\alpha)^{-(l-j)}
   = |I_{r_j+1}| - 2\alpha(d-\alpha)^{-(l-j)} \geq \dfrac{|I_{r_j(\sigma,y)+4}|}{2}.
 \end{split}
 \end{equation*}
 So, it follows from Lemma \ref{lemma34} that $$|\varphi^j(\left\{\theta\right\}\times J)| \geq \dfrac{\delta_0}{2}\cdot \alpha^{\frac{1}{D-1} \left(1-\frac{2\eta}{D-1}\right)} e^{-(c+\varepsilon)(l-j)}.$$
 From this estimate, the fact that $l-j \leq n-j$ and from (\ref{eq20}) we  get 
  \begin{equation*}
  	|\varphi^n(\left\{\theta \right\}\times J)| \geq \dfrac{\delta_0}{2}\cdot \alpha^{\frac{1}{D-1} \left(1-\frac{2\eta}{D-1}\right)} \cdot \exp\left( ((D-1)c-2\varepsilon)(n-j) \right), 
  \end{equation*}
which is the same estimate we obtained in Case 1.
\end{proof}

\subsection{Bounded Distortion}

Let $\mathcal{R} = \bigcup_{n\geq p} \mathcal{R}_{n}$ be the partition of $\mathbb{S}^1\times I$ by the rectangles constructed before and let $h: \mathcal{R}\rightarrow \mathbb{Z}_{+}$ be the map defined as $h(R) = n$, if $R\in\mathcal{R}_n$. Consider the map $\phi: \mathbb{S}^1\times \mathcal{M} \rightarrow \mathbb{S}^1\times \mathcal{M}$ defined in each rectangle $R\in\mathcal{R}$ as $\phi|_R = \varphi^{h(R)}|_R$.

By Corollary \ref{cor36}, $\phi$ maps the interiors of rectangles $R\in\mathcal{R}$ diffeomorphically onto its image. We prove now that the distortion caused by the map $\phi$ in these rectangles are uniformly bounded by constant.  

For what follows, we consider $\varphi^n(\theta,x) = (g^n(\theta), F_n(\theta,x))$, for all $n\geq 1$ and $(\theta,x) \in \mathbb{S}^1\times \mathcal{M}$, which is a consequence of (\ref{cond1}).  

\begin{lema}
	
There is a constant $C>0$  such that for every $(\theta,x)\in \mathbb{S}^1\times \mathcal{M}$  and $n \geq 1$ we have $$\left| \dfrac{\partial_\theta F_n(\theta,x)}{\partial_\theta g^n(\theta)} \right| \leq C.$$
\label{lemma39}
\end{lema}

\begin{proof}
Observing that $|\partial_xf(\theta,x)| = |h'_D(x)|$, we have that $|\partial_xf(\theta,x)|\leq 7/4$ if $x \in I''$ and $7/4 \leq |\partial_xf(\theta,x)| \leq 2$ if  $x\in I'\setminus I''$. For $x \in \mathcal{M}\setminus I'$, we have $|\partial_xf(\theta,x)| = 2$ if $D$ is odd, or $|\partial_xf(\theta,x)| \leq 4$ if $D$ is even. In any case, we have $|\partial_xf(\theta,x)| \leq 4$ and the rest of proof follows exactly as in   \cite[Lemma 4.1]{Al00}.	 	
\end{proof}

For what follows, we need an estimate for $|\partial_x f(\theta,x)|$ near the critical point $\tilde{x}$. 
We claim 
	\begin{equation}
		|\partial_x f(\theta,x)| \geq (DA - \alpha) |x-\tilde{x}|^{D-1},
		\label{cond2'} 
	\end{equation}
	for $|x-\tilde{x}|< \sqrt[D]{\alpha}$ and $\alpha$ small enough.
	
Indeed, let $k>1$ the least integer such that $\partial_x f(\theta,\tilde{x}) \neq 0$.
	Note that (\ref{alpha-nbd}) implies that $k \le D$.
	If $k = D$ then (\ref{cond2'}) follows immediately from  (\ref{alpha-nbd}) and (\ref{cond1}).
	If $k<D$, there is a positive constant $C>0$ such that $|\partial_x f(\theta,x)| \ge Ck|x-\tilde{x}|^{k-1}$ for all $|x-\tilde{x}|< \sqrt[D]{\alpha}$, and $\alpha$ small enough.
	Then
	\begin{align*}
		|\partial_x f(\theta,x)| & \ge Ck|x-\tilde{x}|^{k-1} \ge ((D-\alpha) \sqrt[D]{\alpha^{D-k}}) |x - \tilde{x}|^{k-1} \\
		& \ge ((D-\alpha)|x - \tilde{x}|^{D-k}) |x - \tilde{x}|^{k-1} \ge (DA - \alpha) |x-\tilde{x}|^{D-1}.
	\end{align*}

\begin{prop}
There is a constant $\tilde{C} = \tilde{C}(\alpha) > 0$ such that for every $n \geq p$, $R\in \mathcal{R}_n$ and $(\sigma,y) \in\varphi^n(R)$ we have $$\dfrac{\left\| D(J \circ \phi^{-1} )(\sigma,y) \right\|}{|(J\circ \phi^{-1})(\sigma,y)|} \leq \tilde{C},$$ where $J(\theta,x)$ is the jacobian of $\phi = \varphi^n|_R: R\rightarrow \varphi^n(R)$.
\label{prop310}
\end{prop}

\begin{proof}
We start by observing that $$(D\phi \circ \phi^{-1})(\sigma,y) = \left[ \begin{array}{cc}
	\partial_\theta g^n(\theta) & 0 \\
	\partial_\theta F_n(\theta,x) & \partial_x F_n(\theta,x)
\end{array} \right],$$ from which we get $$[D\phi(\theta,x)]^{-1} = \dfrac{1}{J(\theta,x)} \left[ \begin{array}{cc}
\partial_x F_n(\theta,x) & 0 \\
  -\partial_\theta F_n(\theta,x) & \partial_\theta g^n(\theta)
\end{array} \right].$$
Since $D(J\circ \phi^{-1})(\sigma,y) = DJ(\theta,x) \circ [D\phi(\theta,x)]^{-1}$ it follows that the quantity $\dfrac{\left\| D(J \circ \phi^{-1} )(\sigma,y) \right\|}{|(J\circ \phi^{-1})(\sigma,y)|}$ is equal to $$\dfrac{\left\|(\partial_{\theta} J(\theta,x)\partial_x F_n(\theta,x) - \partial_x J(\theta,x) \partial_{\theta} F_n(\theta,x), \partial_{x} J(\theta,x)\partial_\theta g^n(\theta))\right\|}{[J(\theta,x)]^2}.$$

Therefore, taking Lemma \ref{lemma39} into account,  it suffices to find upper bounds for $$A_1 = \dfrac{\left\| \partial_\theta J(\theta,x) \partial_x F_n(\theta,x) \right\|}{[J(\theta,x)]^2}  \ \ \mbox{and} \ \ A_2 = \dfrac{\left\| \partial_{x} J(\theta,x)\partial_\theta g^n(\theta)) \right\|}{[J(\theta,x)]^2}.$$


Folowing the same estimates as in \cite[Proposition 4.2 ]{Al00}, we get $$A_1 \leq \sum_{j=0}^{n-1} \dfrac{K}{(d-\alpha)(d-\alpha)^{n-j}} + \sum_{j=0}^{n-1} \dfrac{K}{|\partial_x f(\theta_j,x_j)|(d-\alpha)^{n-j}},$$ where $$K > \displaystyle\max\left\{|\partial^2_{\theta\theta}g(\theta)|, |\partial^2_{xx} f(\theta,x)|, |\partial_\theta\partial_x f(\theta,x)| \right\},$$ for any $(\theta,x)\in \mathbb{S}^1\times I$. 

The first sum clearly has an upper bound since $d-\alpha > 1$. For the second sum, from (\ref{cond2'}) we have $$|\partial_xf(\theta_i,x_i)| \geq (DA-\alpha)|x-\tilde{x}|^{D-1} \geq (DA-\alpha)\sqrt[D]{\alpha^{D-1}} e^{-(D-1)r_j(\theta,x)},$$
 for $i=0,1,\ldots, n-1$. Also from ($I_n$) and Lemma \ref{lemma35}, there is $(\sigma,y) \in R\cap H_n$ such that $r_j(\theta,x) \leq r_j(\sigma,y)+4$, for $i=0,1,\ldots, n-1$. Combining these two facts, we get $$\sum_{j=0}^{n-1} \dfrac{K}{|\partial_x f(\theta_j,x_j)|(d-\alpha)^{n-j}} \leq \sum_{j=0}^{n-1} \dfrac{K}{(DA-\alpha)\sqrt[D]{\alpha^{D-1}}\cdot e^{-(D-1)[r_j(\sigma,y)+4]}(d-\alpha)^{n-j}}.$$
 
 Taking $C_1(\alpha) = (DA-\alpha)^{-1}K\alpha^{\frac{D-1}{D}} e^{4(D-1)} > 0$, the sum above becomes $$C_1(\alpha) \left( \sum_{j\in G_n(\sigma,y)} \dfrac{1}{e^{-(D-1)r_j(\sigma,y)}(d-\alpha)^{(n-j)}} + \sum_{j\notin G_n(\sigma,y)} \dfrac{1}{e^{-(D-1)r_j(\sigma,y)}(d-\alpha)^{(n-j)}} \right).
 $$
 
 For $j\in G_n(\sigma,y)$, since $n$ is hyperbolic return for $(\sigma,y)$, we have $r_j(\sigma,y) \leq \dfrac{1}{D-1}(c+\varepsilon)(n-j)$; also since $e^{c+\varepsilon} < d-\alpha$, it follows that the first sum has an upper bound. For $j\notin G_{n}(\sigma, y)$, we have $r_j(\sigma,y) \leq \dfrac{1}{D-1} \left(\dfrac{1}{D} - \dfrac{2\eta}{D-1} \right)\log\left(\dfrac{1}{\alpha}\right)$, which implies that the second sum also has an upper bound. 
 
 A part of the calculations in \cite[Proposition 4.2 ]{Al00} gives the following general estimate of $A_2$:
 $$A_2 \leq \sum_{j=0}^{n-1}\dfrac{K}{|\partial_x f(\theta_j,x_j)| \prod_{i=j}^{n-1} |\partial_xf(\theta_i,x_i)|}.
 $$
 
 By ($I_n$) and Lemma \ref{lemma35}, there is some $(\sigma,y)\in R\cap H_n$ such that $r_j(\theta,x) \leq r_j(\sigma,y)+4$, for $j=0,1,\ldots, n-1$. Then, from Lemma \ref{lemma23}, we get $$\prod_{i=j}^{n-1} |\partial_{x}f(\theta_i,x_i)| \geq \exp\left( (D+1)cn - (D-1)\sum_{i\in G_{n-j}(\sigma_j,y_j)} r_i(\sigma_j,y_j) \right).$$
 
 Since $n$ is a hyperbolic return for $(\sigma,y)$, $n-j$ is a hyperbolic return for $(\sigma_j,y_j)$ and then 

 $$\prod_{i=j}^{n-1} |\partial_{x}f(\theta_i,x_i)| \geq \exp\left((Dc-\varepsilon)(n-j) \right).$$

Finally, since $$|\partial_x f(\theta_j,x_j)| \geq (DA-\alpha)\sqrt[D]{\alpha^{D-1}} e^{-(D-1)[r_j(\sigma,y)+4]},$$ taking $C_2(\alpha) = (DA-\alpha)^{-1}K\alpha^{\frac{D-1}{D}} e^{4(D-1)}>0$ we get $$A_2 \leq  C_2(\alpha)\sum_{j=0}^{n-1} \dfrac{1}{\exp\left((Dc-\varepsilon)(n-j)-r_j(\sigma,y)\right)},$$ and we obtain the upper bound for $A_2$ by splitting the sum for $j\in G_n(\sigma,y)$ and $j\notin G_n(\sigma,y)$, and then finding separate upper bounds in each case using the same reasoning as we did for $A_1$. 

\end{proof}

\section{Proof of Theorem A}

The proof relies on the following theorem from  \cite[Section 5]{Al00}:

\begin{teo}(\cite[Theorem 5.2]{Al00})
Let $\phi: R\rightarrow R$ be a $C^2$ piecewise expanding map with bounded distortion and $\left\{R_i \right\}_{i=1}^{\infty}$ its domain of smoothness. Assume that there are $\beta, \rho>0$ such that each $\phi(R_i)$ has a regular collar with $\beta(\phi(R_i))>\beta$ and $\rho(\phi(R_i))>\rho$. If $\sigma\left(1+1/\beta \right)\leq 1$, then $\phi$ has an absolutely continuous invariant probability measure. 
\label{teo41}   	
\end{teo}

Here a $C^2$ \emph{piecewise expanding map} with \emph{bounded distortion} is map $\phi:R\rightarrow R$ satisfying the following conditions:

\vspace{0.2cm}

\begin{itemize}
 \item[(E1)] There is a partition $\left\{R_i \right\}_{i=1}^\infty$ of $R$ such that each $R_i$ is closed domain with piecewise $C^2$ boundary and finite $(n-1)$-dimensional measure;
  
 \item[(E2)] Each $\phi_i = \phi|{R_i}$ is $C^2$ bijection from the interior of $R_i$ onto its image and has $C^2$ extension to the boundary;
 
 \item[(E3)] There is some $0<\sigma< 1$ such  that $\left\| D\phi_i^{-1} \right\| < \sigma$, for every $i \geq 1$.
 
\item[(D)] The map $\phi$ has \emph{bounded distortion} if there is a constant $K>0$ such that for every $i\geq 1$ $$\dfrac{\left\|D(J\circ \phi_i^{-1}) \right\|}{|J\circ \phi_i^{-1}|} < K,$$ where $J$ is the jacobian of $\phi$.

\end{itemize}

\vspace{0.2cm}

Let $N$  be a closed region in $\mathbb{R}^n$ with piecewise $C^2$ boundary $\partial N$ of finite $(n-1)$-dimensional measure.  We say that a neighborhood $U$ of $\partial N$ in $N$ is a \emph{regular collar} for $N$ if there is a $C^1$ unitary vector field $H$ in $\partial N$ and numbers $\beta(N),\rho(N)>0$ such that:

\begin{itemize}
\item[(C1)] $U$ is written as the union of the segments joining $x\in \partial N$	and $x+\rho(S)H(x) \in N$;

\item[(C2)] For every $x\in\partial N$ and $v\in T_x\partial N$, the angle between $H(x)$ and $v$ are bounded away from zero, with $\sin\theta(x) \geq \beta(N)$, where $\theta(x)$ is the angle between $v$ and $H(x)$.
	
\end{itemize}

For the points $x\in\partial N$ where it fails to be smooth, we define $H(x)$ to be the $C^1$ extension of $H$ to the boundary point $x$. Moreover, the tangent spaces at these points will be considered as the union of the tangent spaces at $x$ of each smooth component of $\partial N$ it belongs to.

\begin{prop}
Let $\phi: \mathbb{S}^1\times \mathcal{M} \rightarrow \mathbb{S}^1\times \mathcal{M}$ be the map defined by $\phi|R = \varphi^{h(R)}|R$, for $R\in\mathcal{R}$. Then, for $p\geq 0$ large enough, the map $\phi$ is a $C^2$ piecewise expanding map with bounded distortion. Moreover $\phi(R)$ admits a regular collar, for all $R\in\mathcal{R}$.
\label{prop42}	
\end{prop}

\begin{proof}
For property (E1), we take the partition $\left\{R_i \right\}_{i=0}^{\infty}$ as the set $\mathcal{R}$ constructed in Section 3. By Corollary \ref{cor33}, the boundary of each rectangle $R\in\mathcal{R}$ will have finite measure. Property (E2) follows from Corollary \ref{cor36}, where we have $\phi|\mbox{int}(R) = \varphi^{h(R)}|\mbox{int}(R)$ and extend it to the boundary of $R$. The bounded distortion property follows from Proposition \ref{prop310}.

Suppose $R\in\mathcal{R}_n$. To check (E3), we first observe that 
\begin{equation*}
	\begin{split}
[D\phi(\theta,x)]^{-1} & = \dfrac{1}{J(\theta,x)} \left[ \begin{array}{cc}
	\partial_x F_n(\theta,x) & 0 \\
	-\partial_\theta F_n(\theta,x) & \partial_\theta g^n(\theta)
\end{array} \right]\\ 
 & = \left[ \begin{array}{cc}
[\partial_\theta g^n(\theta,x)]^{-1} & 0 \\
-\partial_\theta F_n(\theta,x)[\partial_\theta g^n(\theta)\partial_xF_n(\theta,x)]^{-1} & [\partial_x F_n(\theta,x)]^{-1}
\end{array} \right].
\end{split}
\end{equation*}

By Lemma \ref{lemma39}, $$\left\| D\phi(\theta,x)^{-1} \right\| \leq \max \left\{|\partial_\theta g^n(\theta)|^{-1}+C|\partial_xF_n(\theta,x)|^{-1}, |\partial_xF_n(\theta,x)|^{-1} \right\}.$$

We have $\left| \partial_\theta g^n(\theta) \right|^{-1} \leq (d-\alpha)^{-n}$ and by Lemmas \ref{lemma23} and \ref{lemma35}, 
$$
|\partial_xF_n(\theta,x)| \geq \exp\left((Dc-\varepsilon)n \right),
$$ 
which implies 
$$
|\partial_xF_n(\theta,x)|^{-1} \leq \exp\left(-(Dc-\varepsilon)n \right).
$$	
Then 
\begin{equation}
\left\| D\phi(\theta,x)^{-1} \right\| \leq (d-\alpha)^{-n}+(1+C)\exp(-(Dc-\varepsilon)n),
\label{eq24}
\end{equation} 
which can be made smaller then $1$ by taking $p$ large enough (recall that $n \geq p$ in our construction). This proves that $\phi$ is a $C^2$ piecewise expanding map.

To prove that $\phi(R)$ admits a regular collar for all $R\in\mathcal{R}$, we first observe that by Corollary \ref{cor33} the boundary of $\phi(R)$ is made up by two vertical lines and and two admissible curves. Since $|X'(\theta)|\leq \alpha$ for any admissible curve and $\alpha$ is small, it follows that the angles at which the vertical lines meet with the admissible curves in the boundary of $\phi(R)$ have angles uniformly bounded away from zero by a constant $\beta(\phi(R))>0$. This takes care of (C2). 

For (C1), Proposition \ref{prop38} implies that the images $\phi(R)$  has large size for any $R\in\mathcal{R}$. Therefore, we can find some uniform constant $\rho(\phi(R))> 0$ such that the union of segments from $x\in \partial \phi(R)$ to $x+\rho(\phi(R))H(x)$ defines a regular collar in $\phi(R)$, for some $C^1$ unitary vector field $H$ defined on $\partial \phi(R)$ pointing inside of $\phi(R)$. 
\end{proof}

To finish the proof of Theorem A, we will establish the values of $p$ such that there exists a $\phi$-invariant measure $\mu$, which will be guaranteed by Theorem \ref{teo41}, and then define a $\varphi$-invariant measure $\mu^{*}$ induced by $\mu$. The proof that $\mu^{*}$ is invariant, absolutely continuous with respect to the Lebesgue measure, finite and ergodic follows exactly as in \cite[Section 6]{Al00}.

In view of estimate (\ref{eq24}) and the requirements of Theorem \ref{teo41}, it suffices to take $p$ large enough such that 
\begin{equation}
\left[ (d-\alpha)^{-p} +\exp(-(Dc-\varepsilon)p) \right]\left(1 + \dfrac{1}{\beta} \right) < 1.
\label{pcond}
\end{equation}

This proves the existence of an absolutely continuous invariant probability measure $\mu$ for $\phi$ by Theorem \ref{teo41}. 

Now consider the sequence $R_1 = \ldots = R_{p-1} = \emptyset$ and $$R_n = \bigcup_{R\in\mathcal{R}_n} R, \ \ \mbox{for}  \ \ n\geq p.$$

We define $$\mu^{*} = \sum_{j=0}^{\infty} \varphi^{j}_{*} \left( \mu\left|\bigcup_{n \geq j} R_n \right. \right),$$ which is a $\varphi$-invariant measure absolutely continuous with respect to the Lebesgue measure.

Finally to extend this construction to any $\varphi\in \mathcal{N}$, we argue in the same way as in \cite[Section 7 ]{Al00} to replace condition (\ref{cond1}) with the existence of a $\varphi$-invariant foliation $\mathcal{F}$ close to vertical lines, which will give a notion of expansion in the direction of the leaves of $\mathcal{F}$ in place of the derivative. The rectangles in the partition $\mathcal{R}$ will now have their boundaries made up of two admissible curves and two segments of leaves in $\mathcal{F}$, instead of two vertical lines. From here, everything follows exactly as before.

\section{Proof of Theorem B}

From  of \cite[Section 5]{Al00} it follows that we can decompose $\mathbb{S}^1\times \mathcal{M}$ into finitely many minimal $\varphi$-invariant subsets $A$ with positive Lebesgue measure such that there exists finitely many SRB measures $\mu_A$ giving full weight to $A$, see \cite[Section 6]{Al00}.
To prove uniqueness, as described in \cite{AV02}, it suffices to prove that $\varphi$ is topologically mixing. Then it follows from the arguments in  \cite[Section 7]{AV02} that any $\varphi\in\mathcal{N}$ is ergodic with respect to the Lebesgue measure, which is the content of Theorem C of that paper. Therefore, the SRB measure defined for each $\varphi \in\mathcal{N}$ is unique.

Let $\varphi\in\mathcal{N}$ and consider the change of coordinates in $\mathbb{S}^1\times \mathcal{M}$ in which the invariant central leaves are represented by vertical lines $\left\{\sigma= \ \mbox{constant}\right\}$. Under this change of coordinates the maps in $\mathcal{N}$ assume the form $$\varphi(\sigma,y) = (g(\sigma),f_{\sigma}(y)),$$ where $g$ is now only continuous, the family of maps $(f_{\sigma})_{\sigma \in \mathbb{S}^1}$ depends continuously on $\sigma$ and each $f_{\sigma}$ is at least $C^2$. Moreover, by continuity, each $f_\sigma$ is $C^2$-close to $h_D$. All the following arguments are based on the maps $\varphi \in \mathcal{N}$ in these coordinates.

The \emph{attractor} $\Lambda$ of a map $\varphi\in \mathcal{N}$ is defined as the intersection of all forward images of $\mathbb{S}^1\times\mathcal{M}$: $$\Lambda = \bigcap_{n\geq 0} \varphi^n(\mathbb{S}^1 \times \mathcal{M}).$$ 

When $D$ is odd, we have $\Lambda = \mathbb{S}^1\times \mathbb{S}^1 = \mathbb{T}^2$. When $D$ is even, we have the following result:

\begin{lema}

When $D$ is even, $\Lambda$ coincides with $\varphi^2(\mathbb{S}^1\times I)$ if the interval $I$ is properly chosen.
\label{lemma51}
\end{lema}

\begin{proof}
Let $J\subset (-2,2)$ such that $h_D(J)\subset \interior(I)$ (see \cite{HMS24}). Then, by continuity we may take $I$ slightly larger than $J$ such that on the first iterate we have $h_D(I) \subset \interior(I)$ and $h_D^2(I) = J$. Since $f_\theta$ has a critical point at $x=0$ and is $C^2$-close to $h_D$, for all $\theta$, it follows that $f_\theta$ also has a forward invariant interval $\tilde{J}\subset I$ such that $f_\theta(I) \subset \interior(I)$ and $f_\theta^2(I) = \tilde{J}$. So it follows that $$\varphi^2(\left\{\theta\right\}\times I ) = \left\{g^2(\theta)\right\}\times \tilde{J} = J(\theta).$$

Proceeding by induction, it follows that $\varphi^n(\left\{\theta\right\}\times I)$ coincides with $J(g^{n-2}(\theta))$. Now, for any $\theta \in \mathbb{S}^1$ and $n\geq 2$ we have $$\varphi^n(\mathbb{S}^1\times I)\cap \left(\left\{\theta\right\}\times I \right) = \bigcup_{\tau} J(\tau) = \varphi^2(\mathbb{S}^1 \times I)\cap \left(\left\{\theta\right\}\times I\right),$$ where union above is taken over all $\tau \in \mathbb{S}^1$ such that $g^2(\tau) = \theta$. Therefore $\Lambda = \bigcap_{n\geq 0} \varphi^n(\mathbb{S}^1\times I) = \varphi^2(\mathbb{S}^1 \times I)$.

\end{proof}

We say that a map $\varphi\in \mathcal{N}$ is \emph{topologically mixing} if for any open set $A \subset \mathbb{S}^1\times\mathcal{M}$ there exists a positive integer $n = n(A)$, depending on $A$, such that $\varphi^n(A) = \Lambda$. We will show that this is true for any $R$ in the partition $\mathcal{R}$ of $\mathbb{S}^1 \times I$ constructed in Section 3.

\begin{prop}
There is an integer $M = M(\alpha)$ such that $\varphi^{h(R)+M}(R) = \Lambda$,  for any $R\in\mathcal{R}$.  
\label{prop52}
\end{prop}

\begin{proof}
We follow the same arguments as in  \cite[Proposition 6.2]{AV02}.

We split the proof in four steps. First prove that $|\varphi^{h(R)}(R)| \geq C\cdot \alpha^{1-\frac{2\eta}{D-1}}$, for any $R\in\mathcal{R}$. After that, we show that after a finite number of iterates $n$ we can make $|\varphi^n(R\cap (\left\{\theta\right\}\times\mathcal{M}))| \geq C \cdot \alpha^\frac{1}{D}$, where $C$ is a constant. On the third step, we show again that after a finite number iterates $m$, starting from $J = \varphi((R\cap (\left\{\theta\right\}\times\mathcal{M}))$, we can make $|\varphi^m(J)|$ larger than some constant independent of $\alpha$. Finally, on step four, using the fact that $h_D$ is topologically mixing and that $f_\sigma$ is close to $h_D$ we conclude the proof.   

\vspace{0.2cm}

\noindent {\bfseries Step 1:} There is a constant $\delta_1 > 0$ such that for every $R=\omega\times J \in \mathcal{R}$ and $\theta\in\omega$, $$|\varphi^{h(R)}(\left\{\theta\right\}\times J)| \geq \delta_1 \cdot \alpha^{\frac{1}{D-1} \left(1-\frac{2\eta}{D-1}\right)}.$$ 

This is a direct consequence of Proposition \ref{prop38}.

\vspace{0.2cm}

\noindent {\bfseries Step 2:} There is a constant $\delta_2> 0$ and $M_1 = M_1(\alpha)$ such that, given any $\theta\in \mathbb{S}^1$ and $J \subset \mathcal{M}$  with $|J| \geq \delta_1 \cdot \alpha^{\frac{1}{D-1} \left(1-\frac{2\eta}{D-1}\right)}$, there is an $n\leq M_1$ such that $$|\varphi^n(\left\{\theta \right\}\times J)| \geq \delta_2 \cdot \alpha^{\frac{1}{D}}.$$ 

Take $\delta_2 = 1$. Let $R_0\geq 0$ be the first integer for which $f_\theta^{R_0}(J)$ intersects $\left( \tilde{x}-\sqrt[D]{\alpha},\tilde{x}+\sqrt[D]{\alpha} \right)$. By Lemma \ref{lema21'}, the iterates of $|J|$ grows exponentially fast until iterate $R_0$. Since $|J|$ is bounded below by a power of $\alpha$, it follows that $R_0 \leq \tilde{C} \cdot \log \left(1/\alpha\right)$. Now we have two possible scenarios, the first one being that $f_\theta^{R_0}(J)$ is not contained in $\left(\tilde{x}-2\sqrt[D]{\alpha},\tilde{x}+ 2\sqrt[D]{\alpha} \right)$, in which case it follows that $|f_\theta^{R_0}(J)| \geq \sqrt[D]{\alpha}$ and we take $n = R_0$. If, however, $f_\theta^{R_0}(J)\subset \left(\tilde{x}-2\sqrt[D]{\alpha},\tilde{x}+ 2\sqrt[D]{\alpha} \right)$, the by Lemma \ref{lema21'} gives $$|f_\theta^{R_0}(J)| \geq C_2\tau^{R_0} |J| \geq C_2 |J| \geq C_2\delta_1 \alpha^{\frac{1}{D-1} \left(1-\frac{2\eta}{D-1}\right)}.$$

Let $J_1\subset f_\theta^{R_0}(J)$ be such that $$J_1\cap \left(\tilde{x}-\dfrac{C_2\delta_1}{4}\alpha^{\frac{1}{D-1} \left(1-\frac{2\eta}{D-1}\right)},\tilde{x}+ \dfrac{C_2\delta_1}{4}\alpha^{\frac{1}{D-1} \left(1-\frac{2\eta}{D-1}\right)} \right) = \emptyset$$ and $|J_1|\geq \dfrac{C_2}{4}|J|$, and let $\sigma_1 = g^{R_0}(\theta)$.	By Lemma \ref{lema21'}, there exists $N = N(\alpha) \leq K_1 \cdot \log(1/\alpha)$, where $K_1$ is constant, such that 
\begin{equation}
|f_{\sigma_1}^N(J_1)| \geq \dfrac{C_2^{D-1}\delta_1^{D-1}}{4^{D-1}}\alpha^{ 1-\frac{2\eta}{D-1}} \alpha^{-1+\frac{\eta}{D-1}}|J_1| \geq \dfrac{C_2^{D-1}\delta_1^{D}}{4^D} \alpha^{-\frac{\eta}{D-1}} |J|.
\label{eq25}
\end{equation}

Now we take $\alpha$ small enough such that the right side of (\ref{eq25}) is larger than $2|J|$. We can now repeat this process with $\sigma_2 = g^{R_0+N}(\sigma)$ and $J_2 = f_\sigma^{R_0+N}(J)$ in place of $J$. In this way, we construct a sequence of vertical segments $J_0 = J, J_2,\ldots, J_{2l}$; a sequence of points $\sigma_0 = \theta,\sigma_2,\ldots, \sigma_{2l}$ in $\mathbb{S}^1$, and a sequence of integers $R_0, R_2,\ldots, R_{2l}$ such that $$|J_{2j+2}| > 2|J_{2j}| \ \ \ \mbox{and} \ \ \ J_{2j+2}\subset f_{\sigma_{2l}}^{R_{2j}+N}(J_{2j}),$$  for every $0\leq j < l$.  Since the lengths of the intervals $J_{2j}$ doubles at each step, we will eventually reach a situation where $J_{2l+1} = f_{\sigma_{2l}}^{R_{2l}}(J_{2l})$ is not contained in $\left(\tilde{x}-2\sqrt[D]{\alpha},\tilde{x}+ 2\sqrt[D]{\alpha}\right)$, which implies $|J_{2l+1}|\geq \sqrt[D]{\alpha}$. We then take $n = R_0+N+R_2+n+\ldots+N+R_{2l}$. Since $|J_{2j}|$ increases exponentially fast, it follows that  $l\leq K_2 \cdot \log(1/|J|) \leq K_3 \cdot \log(1/\alpha)$, where $K_2$ and $K_3$ are constants. This together with the fact that $R_j$ and $N$ are also bounded by $K_4 \cdot \log(1/\alpha)$, we get $n\leq K_5 \log^2(1/\alpha)$, where $K_4$ and $K_5$ are also constants. So it suffices to take $M_1(\alpha) = K_5\cdot \log^2(1/\alpha)$.  

\vspace{0.2cm}

\noindent {\bfseries Step 3:} There is a constant $\delta_3 >0$ and an integer $M_2 = M_2(\alpha)$ such that, given any $\theta\in \mathbb{S}^1$  and any interval $J\subset \mathcal{M}$ with $|J| \geq \delta_2 \alpha^{\frac{1}{D}}$, there exists $n\leq M_2$ such that $$|\varphi^n(\left\{\theta\right\}\times J)| \geq \delta_3.$$ 

Following the same arguments in Step 2, we obtain the following estimate, which is an analogue of (\ref{eq25}):

$$|f_{\theta}^{R_0+N}(J)| \geq \dfrac{C_2^{D-1}\delta_2^{D-1}}{4^{D-1}}\alpha^{\frac{D-1}{D}}\alpha^{-1 + \frac{\eta}{D-1}} |J_1| \geq \dfrac{C_2^{D-1}\delta_2^D}{4^D} \alpha^{\frac{\eta}{D-1}}.$$

Let $R_1$ be the first integer such that $f_\theta^{R_0+N+R_1}(J)$ intersects  $(\tilde{x}-\sqrt[D]{\alpha},\tilde{x}+\sqrt[D]{\alpha})$. Now fix small constants $0 < \gamma_1 \leq \gamma_0\leq \gamma$, independent of $\alpha$. If $f_\theta^{R_0+N+R_1}(J)\not\subset(\tilde{x}-\gamma_1,\tilde{x}+\gamma_1)$ then $|f_\theta^{R_0+N+R_1}(J)| > \gamma_1-\sqrt[D]{\alpha} > \gamma_1/2$, and the result follows. If, however, $f_\theta^{R_0+N+R_1}(J) \subset (\tilde{x}-\gamma_1,\tilde{x}+\gamma_1)$ then we apply Lemma \ref{lema21'} to obtain $$|f_\theta^{R_0+N+R_1}(J)| \geq C_2 \tau^{R_1}\dfrac{C_2^{D-1}\delta_1^D}{4^D} \alpha^{\frac{\eta}{D-1}} \geq 4C_3\alpha^{\frac{\eta}{D-1}},$$ where $C_3 = \dfrac{C_2^D\delta_1^D}{4^{D+1}}$. Then, there is some connected component $$\tilde{J} \subset f_\theta^{R_0+N+R_1}(J) \setminus \left(\tilde{x}-\sqrt[D]{\alpha},\tilde{x}+ \sqrt[D]{\alpha}\right)$$ whose length is larger than $2C_3\alpha^{\frac{\eta}{D-1}} - \alpha^{\frac{1}{D}} > C_3 \alpha^{\frac{\eta}{D-1}}$, since $\alpha$ is small and $\eta < 1/4$. 

Let $\tilde{\sigma} = g^{R_0+N+R_1}(\theta)$, $l\geq 1$ be the smallest integer such that $z =h_D^l(\tilde{x})$ is a periodic point, $k\geq 1$ be the period of $z$ and $\rho^k = |(h_D^k)'(\tilde{x})|$. From \cite{HMS24}, we know that $\rho > 1$. Fix $\rho_1,\rho_2 > 0$ such that $\rho_1 \leq \rho\leq \rho_2$ and $\rho_1^{1-\frac{\eta}{D}} > \rho_2$, and take $\gamma_0>0$ small enough such that $$\rho_1^k \leq |Df^k_\theta(y)| \leq \rho_2^k, \ \ \ \ \mbox{whenever} \ \ \ \ |y-z|< \gamma_0,$$ for any $\theta \in \mathbb{S}^1$, which is possible since $f_\theta$ is $C^2$-close to $h_D$ for $\alpha$ sufficiently small.    

Since $h_D^j(\tilde{x}) \neq \tilde{x}$, for any $j > 0$, fixing $\gamma_1 >0$ sufficiently small, we can ensure that $f_\theta^j(x)$ remains outside of a fixed neighborhood of $\tilde{x}$, for all $0\leq j \leq l$ and $x \in \tilde{J}$. So, we have $$|Df^l_{\tilde{\sigma}}(y)| \geq \tilde{C} |y-\tilde{x}|^{D-1},$$ for some constant $\tilde{C}>0$ and for all $y \in \tilde{J}$.	 It follows that for some $y\in\tilde{J}$, $$|f^l_{\tilde{\sigma}}(J)| \geq |Df^l_{\tilde{\sigma}}(y)||\tilde{J}| \geq \tilde{C}|y-\tilde{x}|^{D-1} C_3 \alpha^{\frac{\eta}{D-1}} \geq C_0 \alpha^{\frac{D-1}{D}+\frac{\eta}{D-1}},$$ for some constant $C_0>0$.

Given any $x\in\tilde{J}$  and $i\geq 0$, let $d_i = |x_{l+ki} - z|$, where $(\theta_j,x_j) = \varphi^j(\theta,x)$. We take $\gamma_1>0$ small enough such that $$|x- \tilde{x}| \leq \gamma_1 \ \ \Rightarrow \ \ d_0 \leq C|x-\tilde{x}|^D +C\alpha \leq \gamma_0.$$

Now, if $(\theta, x)$ and $i\geq 1$ are such that $|x-\tilde{x}| < \gamma_1$ and $d_1,\ldots, d_{i-1} < \gamma_0$, then by the Mean-value Theorem we have $d_{i} \leq \rho_2^k d_{i-1}+ C\alpha$. So, by induction, 
$$d_i \leq \left(1+\rho_2^k+\ldots + \rho_2^{(i-1)k}\right)C\alpha + \rho_2^{ki}d_0 \leq \rho_2^{ki}\left(C|x-\tilde{x}|^D +C\alpha\right).$$ 

In particular, for $|x - \tilde{x}| \leq \sqrt[D]{\alpha}$ , we have $d_i \leq \rho_2^{ki}C\alpha$. Let $N_0$ be the smallest integer such that $\rho_2^{kN_0} > \gamma_0/2$. So this choice of $N_0$ implies that $d_i \leq \gamma_0/2$, for $i=0,1,\ldots,N_0-1$.

Now, we consider two possible cases:

\noindent \emph{Case 1:} $f_{\tilde{\sigma}}^{l+ki}(\tilde{J}) \subset (z-\gamma_0,z+\gamma_0)$, for every $ i = 1,\ldots,N_0-1$.
	
	\vspace{0.1cm}
	
\noindent	 Recall that $\eta < 1/4$. Then, we have
	 \begin{equation*}
	 \begin{split}
	 |f_{\tilde{\sigma}}^{l+ki}(\tilde{J})| &\geq \rho_1^{ki}|f_{\tilde{\sigma}}^{l}(\tilde{J})| \geq C_0 \rho_2^{(1-\frac{\eta}{D})kN_0} \alpha^{\frac{D-1}{D}+\frac{\eta}{D-1}}
	  \geq K \alpha^{-1+\frac{\eta}{D}}\alpha^{\frac{D-1}{D}+\frac{\eta}{D-1}} \geq K \alpha^{\frac{\eta-1}{D}+\frac{\eta}{D-1}} > 1,
	 \end{split}
	\end{equation*}
	which is absurd, since $\gamma_0$ is small.
	
\noindent	\emph{Case 2:} There is some $1\leq i \leq N_0-1$ such that  $f_{\tilde{\sigma}}^{l+ki}(\tilde{J}) \not\subset (z-\gamma_0,z+\gamma_0)$.

	\vspace{0.1cm}
	
\noindent	In this case, since $d_i\leq \gamma_0/2$, it follows that $$|f_{\tilde{\sigma}}^{l+ki}(\tilde{J})| \geq \gamma_0 -\gamma_0/2 > \gamma_1/2.
$$
Then, we take $\delta_3 = \gamma_1/2$, $n = R_0+N+R_1+l+ki$ and $M_2 = R_0+N+R_1+l+ki+kN_0$. 

\vspace{0.2cm}

\noindent {\bfseries Step 4:} There is an integer $M_3$ such that if $J\subset I$ is an interval with $|J| \geq \delta_3$, then for every $\theta\in \mathbb{S}^1$ we have $$\varphi^{M_3}(\left\{\theta\right\}\times J) =  \left(\left\{g^{M_3}(\theta)\right\}\times I \right)\cap \Lambda.$$  

\vspace{0.1cm}

Since $h_D$ is $C^2$ its critical point $\tilde{x}$ is non-flat, i.e., some higher order derivative of $h_D$ at $\tilde{x}$ is nonzero, then it follows from Theorem A of Chapter IV in \cite{MS93} that $h_D$ has no wandering intervals. In particular, since the critical point of $h_D$ was chosen to be pre-periodic, it follows that the pre-orbit of the repelling fixed point $z$ is dense in $I$. Thus, we can find an integer $n_1 \geq 0$ such that $h_D^{-n_1}(z)$ intersects all intervals with length $\delta_3/2$, which implies that the image of any $J \subset I$ with $|J|\geq \delta_3$ must contain a neighborhood $J_0$ of $z$ whose size depends on $\delta_3$. After a finite number of iterates $n_2\geq 1$, we must have $h_D^{n_2}(J_0) = J$.  

Let $M_3 = n_1+n_2+1$. By continuity we have $$\varphi^{M_3}(\left\{\theta\right\}\times J) = \left\{g^{M_3}(\theta)\right\}\times J(g^{M_3-2}(\theta)) = \left(\left\{g^{M_3}(\theta)\right\}\times I \right)\cap \Lambda,$$ where $J(\theta)$ is the segment described in the proof of Lemma \ref{lemma51}.

\vspace{0.1cm}

Finally, take $M = M_1+M_2+M_3$ and the result follows.  
\end{proof}

Now we prove that each $\varphi\in\mathcal{N}$ is topologically mixing.

Recall that in the definition of the partition $\mathcal{R}$ in Section 3, we start with a fixed positive integer $p$ large enough to satisfy \eqref{pcond}, with the map $h:\mathcal{R}\rightarrow \mathbb{Z}_{+}$ satisfying $h(R) \geq p$, for any $R\in\mathcal{R}$. By Corollary \ref{cor36} and \eqref{eq24} the diameter $$\mbox{diam}(\mathcal{R}) = \sup\left\{\mbox{diam}(R): \ R\in\mathcal{R} \right\}$$
can be made arbitrarily small by increasing $p$. We will now consider the sequence of partitions $(\mathcal{R}_p)_{p\geq p_0}$ of $\mathbb{S}^1\times \mathcal{M}$ and the maps $h_p:\mathcal{R}_p\rightarrow \mathbb{Z}_{+}$ associated with each $\mathcal{R}_p$, where $p_0$ satisfies \eqref{pcond}.

Given any open set $A\subset \mathbb{S}^1\times I$, since $\mbox{diam}(\mathcal{R}_p) \rightarrow 0$ as $p\rightarrow +\infty$, we can find some $\tilde{p} \geq p_0$ such that there is $R\in \mathcal{R}_{\tilde{p}}$ with $R\subset A$. Fix $\tilde{p}$ and take $M$ as in Proposition \ref{prop52}. Then, there is some $n \leq h(R)+M$ such that $\varphi^n(A) = \Lambda$, which implies that $\varphi$ is topologically mixing.


\begin{thebibliography}{10}

\bibitem{AV02}
J.~Alves and M.~Viana.
\newblock Statistical stability for robust classes of maps with non-uniform
  expansion.
\newblock {\em Ergodic Theory Dynam. Systems}, 22:1--32, 2002.

\bibitem{Al00}
J.~F. Alves.
\newblock {SRB} measures for non-hyperbolic systems with multidimensional
  expansion.
\newblock {\em Ann. Sci. {\'E}cole Norm. Sup.}, 33:1--32, 2000.

\bibitem{AA03}
J.~F. Alves and V.~Ara\'ujo.
\newblock Random perturbations of nonuniformly expanding maps.
\newblock {\em Ast\'erisque}, 286:25--62, 2003.

\bibitem{ABV00}
J.~F. Alves, C.~Bonatti, and M.~Viana.
\newblock S{RB} measures for partially hyperbolic systems whose central
  direction is mostly expanding.
\newblock {\em Invent. Math.}, 140:351--398, 2000.

\bibitem{ALP05}
J.~F. Alves, S.~Luzzatto, and V.~Pinheiro.
\newblock Markov structures and decay of correlations for non-uniformly
  expanding dynamical systems.
\newblock {\em Ann. Inst. H. Poincar\'e{} C Anal. Non Lin\'eaire},
  22(6):817--839, 2005.

\bibitem{MS93}
W.~de~Melo and S.~van Strien.
\newblock {\em One-dimensional dynamics}.
\newblock Springer Verlag, 1993.

\bibitem{HMS24}
V.~Horita, N.~Muniz, and O.~Sester.
\newblock Building expansion for generalizations of {V}iana maps.
\newblock {\em J. Dyn.. Diff. Equat.}, 2024.

\bibitem{Ja81}
M.~Jakobson.
\newblock Absolutely continuous invariant measures for one-parameter families
  of one-dimensional maps.
\newblock {\em Comm. Math. Phys.}, 81:39--88, 1981.

\bibitem{Man87}
R.~Ma\~n\'e.
\newblock {\em Ergodic theory and differentiable dynamics}, volume~8 of {\em
  Ergebnisse der Mathematik und ihrer Grenzgebiete (3) [Results in Mathematics
  and Related Areas (3)]}.
\newblock Springer-Verlag, Berlin, 1987.
\newblock Translated from the Portuguese by Silvio Levy.

\bibitem{Vi98}
M.~Viana.
\newblock Dynamics: a probabilistic and geometric perspective.
\newblock In {\em Proceedings of the International Congress of Mathematicians,
  Vol. I (Berlin, 1998)}, volume 1998, pages 557--578.

\bibitem{Vi97}
M.~Viana.
\newblock Multidimensional nonhyperbolic attractors.
\newblock {\em Inst. Hautes \'{E}tudes Sci. Publ. Math.}, (85):63--96, 1997.

\end{thebibliography}

\bigskip

\flushleft

{\bf Ricardo Chical\'e} (ricardo.chicale\@@gmail.com)\\
Departamento de Matem\'{a}tica, IBILCE/UNESP \\
Rua Crist\'{o}v\~{a}o Colombo 2265\\
15054-000 S. J. Rio Preto, SP, Brazil

\bigskip

\flushleft

{\bf Vanderlei Horita} (vanderlei.horita\@@unesp.br)\\
Departamento de Matem\'{a}tica, IBILCE/UNESP \\
Rua Crist\'{o}v\~{a}o Colombo 2265\\
15054-000 S. J. Rio Preto, SP, Brazil

\end{document}